\documentclass{amsart}

\usepackage{hyperref}
\hypersetup{colorlinks,linkcolor=blue,citecolor=blue}
\usepackage{amssymb}
\usepackage{cite}
\usepackage{amsmath}
\usepackage{comment}
\usepackage{bookmark}
\usepackage{enumerate}
\usepackage{tikz}
\usepackage{a4wide}
  \usepackage{bm}
\newtheorem{thm}{Theorem}[section]
\newtheorem{theorem}[thm]{Theorem}
\newtheorem{lem}[thm]{Lemma}
\newtheorem{definition}[thm]{Definition}

\newtheorem{defi}[thm]{Definition}

\newtheorem{quest}[thm]{Question}

\newtheorem{prop}[thm]{Proposition}
\newtheorem{rem}[thm]{Remark}

\newtheorem{fact}[thm]{Fact} 

\newcommand{\norm}[1]{\left\lVert#1\right\rVert}
\newcommand{\cC}{{\mathcal C}}
\newcommand{\cI}{{\mathcal I}}
\newcommand{\cJ}{{\mathcal J}}

\numberwithin{equation}{section}

\begin{document}

\title[Isomorphic embedding from Lorentz function spaces into Lorentz operator ideals]{
Lack of isomorphic embedding from Lorentz function spaces into Lorentz operator ideals}

\author[J. Huang]{Jinghao Huang}
\address{Institute for  Advanced Study in  Mathematics of HIT, Harbin Institute of Technology, Harbin, 150001, China}
\email{jinghao.huang@hit.edu.cn}
 
\author[F. Sukochev]{Fedor Sukochev}
\address{School of Mathematics and Statistics, University of NSW, Sydney,  2052, Australia}
\email{f.sukochev@unsw.edu.au}
 
\author[D. Zanin]{Dmitriy Zanin}
\address{School of Mathematics and Statistics, Central South University, Changsha, 410083, China}
\email{d.zanin@csu.edu.cn}

\thanks{J. Huang was supported the NNSF of China (No. 12301160,  12471134  and 12671159). F. Sukochev and D. Zanin were supported by the ARC}


\begin{abstract} 
In the present paper, we study the existence of  isomorphic embeddings from Lorentz function spaces $L_{p_1,q_1}$ into Lorentz ideals $\mathcal{C}_{p_2,q_2}$ in $B(H).$  In particular, we show that, for $p,q\in (0,\infty),$ the quasi-Banach function space $L_{p,q}(0,1)$  isomorphically embeds into the ideal $\mathcal{C}_{p,q}$ if and only if $(p,q)=(2,2).$ This extends several existing  results in the literature and answers a question due to Astashkin et. al. \cite{AHS}. 
\end{abstract}

\maketitle

\section{Introduction}
\subsection{Main result}
We say that a (quasi-)Banach space $X$ isomorphically embeds into  a (quasi-)Banach space $Y$ (we write $X\hookrightarrow Y$) if there exists a bicontinuous linear embedding
$T:X\rightarrow Y.$ In other words, $T$ is a continuous linear injection and $T^{-1}:T(X)\rightarrow X$ is also continuous. We write $X\not\hookrightarrow Y$ if $X$ does not isomorphically embed  into $Y.$ 
If $X,Y$ are isomorphic Banach spaces, then we write  $X\approx Y$.  We use \cite{LT1,LT2,AK} as general references to Banach space theory.
 
 %$$d(X,Y)=\inf\{\|T\|\|T^{-1}\|,\  T \ \text{is an isomorphism from} \ X  \ \text{onto}  \ Y\}$$ where the infimum is taken over all possible isomorphisms $T$. 
% If Banach spaces $X$ and $Y$ are not isomorphic, then we write $d(X,Y)=\infty$.

In  \cite{AHS}, the authors study the non-embeddings of symmetric function spaces into operator ideals.
In particular, it was shown that 
$$L_{p,q}(0,1)\not\hookrightarrow \mathcal{C}_{p,q}$$
when $p>q>2$ or $1< p<2$ or $p=2, q\ne 2$. 
In the special setting when $p=q$, this recovers 
Arazy and Lindenstrauss' result in \cite{AL}, where   they proved $L_p(0,1)$ embeds into the Schatten class $C_p$, $1\le p<\infty$, if and only if $p=2$. 
However, the case when $p>2$ and $q\in [1,2]\cup (p,\infty )$ was left as an open question\cite[Question 6.3]{AHS} (see also \cite[Section 5]{HSSZ}):
\begin{quest}\label{quest}
Does $L_{p,q}(0,1) $ isomorphically embed into 
$C_{p,q}$
when $2<p<\infty $ and $q\in [1,2]\cup (p,\infty )$?
\end{quest}
 The main goal of the present paper is to answer  the above question in a   more general setting of quasi-Banach ideals.
\begin{thm}\label{main cor} 
Let $p,q\in (0,\infty).$ We have $L_{p,q}(0,1)\hookrightarrow \mathcal{C}_{p,q}$ if and only if $p=q=2.$
\end{thm}
The approach used in the present paper relies on study of the isomorphisms between function spaces over finite and infinite intervals, see Theorem \ref{AS thm} below  (see also \cite{JMST,LT2} for related results in the setting of Banach symmetric function spaces), which goes back to a question posed by Mityagin \cite[p.99]{Mityagin} (see \cite{HSZ25} for the study of this question in the noncommutative setting). This is markedly different from the approach used in \cite{AL} and \cite{AHS}.
For various results concerning the lack of isomorphic embedding from  operator ideals into Banach lattices, we refer to \cite{Mc,Gillespie,Pisier,KP}.

 \subsection{Background and motivation}

Let $E$ be a (quasi-Banach) symmetric sequence space.
Recall that
there exists a one-to-one correspondence (the so-called Calkin correspondence) between (quasi-Banach) symmetric sequence spaces $E$ and quasi-Banach ideals $\cC_E$ in $B(H)$\cite{KS,LSZ,S14,GK}.
We denote by $\mathcal{C}_p$
the Schatten--von Neumann $p$-class.

  The main motivation of this paper 
  can be traced back to \cite[Theorem 6.1]{LP}, where Lindenstrauss and Pelczynski presented an oversimplified and incomplete proof
for $L_p(0,1)\not\hookrightarrow (\ell_2\oplus \cdots\oplus \ell_2\oplus\cdots)_{\ell_p} $  
   (see a related comment in \cite[p.108]{AL}).
 Later,
  Arazy and Lindenstrauss\cite{AL}, rectifying that oversight, showed that   the   
  function space  $$L_{p}(0,1) \not\hookrightarrow  \mathcal{C}_p$$
 when $1\le p\ne 2 <\infty $, 
 which immediately implies the result claimed in \cite[Theorem 6.1]{LP}. 
 In view of the classical Kadec--Pelczynski theorem for $L_p$, $p>2$\cite{KP70}, %one cannot hope on the existence of symmetric sequence space $G$ such that $G\hookrightarrow L_{p}(0,1)$ and $G\not\hookrightarrow \mathcal{C}_p$. 
 the only symmetric sequence spaces isomorphic to subspaces of $L_p(0,1)$ are $\ell_p$ and $\ell_2$. 
 A deep result in \cite{HOS} demonstrates that a subspace of $L_p(0,1 )$, $p>2$, either isomorphically embeds into $l_p\oplus l_2$ or contains $(\ell_2\oplus \cdots\oplus \ell_2\oplus\cdots)_{\ell_p}$,
both of which  are subspaces of $\mathcal{C}_p$. 
These observations demonstrate that the case when $p>2$ is much harder than the case when $1\le p<2$ due to the fact that   $L_q(0,1) $ embeds $L_p(0,1)$, $1\le p<2$,   isometrically when $1\le p\le q<2$\cite[Proposition 11.1.9]{AK}.

The Lorentz spaces $L_{p,q}$ (introduced by Lorentz in \cite{Lorentz50,Lorentz51}) are important generalizations of $L_p$-spaces. Their importance has been demonstrated in several areas of 
mathematics such as harmonic analysis, interpolation theory, mathematical physics, etc. (see e.g. \cite{BS,Dilworth,CD85} and references therein).
When $(p,q)\in (1,\infty ]\times [1,\infty ]$, 
$L_{p,q}$ can be equipped with an equivalent norm \cite[Chapter 4, Lemma 4.5]{BS}; when $0<p,q<\infty,$ $L_{p,q}$ is separable, see e.g. \cite[Chapter 4]{BS} and \cite{KPS}. For detailed exposition of $L_{p,q}$-spaces,  see \cite{BS} and \cite{Dilworth}. We denote the quasi-Banach ideal (Calkin) corresponding to the $\ell_{p,q}$-sequence space by $\mathcal{C}_{p,q}.$

The study of subspaces of $L_{p,q}$-spaces has gained  attention during past decades.
%Recall that a Banach space $X$ is said to be primary if whenever $X$ isomorphic to $Y\oplus Z$, then either $Y$ or $Z$ is isomorphic to $X$. It is known that $L_{p,q}(0,1)$ and $L_{p,q}(0,\infty)$ are primary spaces  (see \cite[Theorem  2.d.11.]{LT2} and \cite[Theorem A.1]{Dilworth90}) when $1<p<\infty$ and $1\le q<\infty$.
There are several criteria for  a sequence in $L_{p,q}$ to be equivalent to the $\ell_q$-basis (see e.g. \cite[Lemma 2.1]{CD89} and \cite[Proposition 4.e.3]{LT1}). 
This together with a  criterion for a symmetric sequence    generating a complemented subspace in a Banach lattice 
in \cite[Lemma 8.10]{JMST} shows  that  
$\ell_{p,q}$ 
is not isomorphic to a complemented subspace of $L_{p,q}(0,1)$,
and
$L_{p,q}(0,1)\not\approx L_{p,q}(0,\infty)$ when  $1< p<\infty$, $1\le q<\infty$, $p\ne  q $, see
 \cite[Corollary 2.2]{CD89} (see also \cite{Dilworth}).
A stronger result showing that 
$$\ell_{p,q}\not\hookrightarrow L_{p,q}(0,1)$$
 has been proven in \cite{KurS} and \cite{SS} by the so-called ``subsequence splitting lemma'' introduced and studied in \cite{S96} (see also \cite{ASS,HSS} for various applications of this techniques).
As an application,  
if $1<p<\infty$, $1\le q<\infty $,  $p\neq q$,  then the full  list of pairwise non-isomorphic $L_{p,q}$-spaces over a resonant measure space is obtained:
$$\ell_{p,q}^n, ~n=1,2,\cdots, ~\ell_{p,q},\; L_{p,q}(0,1),\; L_{p,q}(0,1)\oplus \ell_{p,q}\;\;\mbox{and}\;\; L_{p,q}(0,\infty).$$
The non-resonant case was treated in \cite{HS21}.
 
In  \cite{AHS}, the authors study the non-embeddings of symmetric function spaces into operator ideals. 
In particular, as mentioned above,   
$$L_{p,q}(0,1)\not\hookrightarrow \mathcal{C}_{p,q}$$
when $p>q>2$ or $1< p<2$ or $p=2, q\ne 2$. 
However, the remaining  cases  were left unanswered (see Question \ref{quest} above).
It is shown in  \cite{Carothers} that when $p>2$ and $1\le q\le p<\infty$, 
any symmetric function space $Y(0,1)$, which is isomorphic to a subspace of $L_{p,q}(0,1)$, is $L_{p,q}(0,1)$ or $L_2(0,1)$ (up to equivalent norms).
This demonstrates 
the paucity of symmetric subspaces in $L_{p,q}(0,1)$  and ensuing difficulty  in studying Question~\ref{quest}. 
In a recent paper\cite{HSSZ}, it is shown that 
$$L_{p,q}(0,\infty)\not\hookrightarrow \mathcal{C}_{p,q}$$
whenever $1<p<\infty$ and $1\le q<\infty $. 
However, since 
$L_{p,q}(0,1)$ is not isomorphic to $L_{p,q}(0,\infty)$, 
Question \ref{quest} remains unresolved.

In the present paper, we consider a question which is more general than Question~\ref{quest} above. 
\begin{quest}\label{quest2}
Let $p_1\ne q_1,p_2\ne q_2\in(0,\infty).$ Does $L_{p_1,q_1}(0,1) $ isomorphically embed into 
$\mathcal{C}_{p_2,q_2}$?
\end{quest}
We answer this question in Theorem~\ref{main thm}  below. Our approach  used in the proof of Theorem~\ref{main thm} is completely
 distinct from the one used in \cite{AL} (and \cite{AHS})
 and its novelty is our main contribution to the study of the geometric structure of ideals $\cC_{p,q}$  (and more broadly that of $\cC_E$). One of the key ingredients in our proof is a deep result concerning isomorphic embedding of a symmetric function space on the semi-axis into a symmetric function space on a finite interval due to Johnson et. al.  \cite[Section 8]{JMST} (see also \cite[Theorem 2.f.1]{LT2}). This theorem asserts that if 
a symmetric function space  
 $E(0,1)$ has a non-trivial Boyd index ${\bm \alpha}_E$, then 
$$Z_E^2 (0,\infty )\hookrightarrow E(0,1),$$
where  $Z_E^2(0,\infty)$  is  
the space of all measurable functions on $(0,\infty )$ such that 
$$\left\|f\right\|_{Z_E^2}:=\left\|\mu(f)\chi_{(0,1)}\right\|_E+\left\|\mu(f)\chi_{(1,\infty)}\right\|_{L_2}<\infty.$$ 
Here, $\mu(f)$ stands for the decreasing rearrangement of a measurable function $f$, see Section \ref{s:pre} below. 
This was recast by Astashkin (and the second-named author) in terms of the Kruglov property\cite{A11,AS04,AS05,AS07,AS-uspehi} (a symmetric function space $E(0,1)$ is said to have the Kruglov property if the Kruglov operator (see e.g. \cite[p.1029]{AS-uspehi} for the definition) acts boundedly on $E(0,1)$\cite[Lemma 7]{AS-uspehi}). 

The following Theorem was proved in \cite{AS-uspehi} for Banach symmetric function spaces (see Theorem 37 there). The proof in the quasi-Banach setting follows that in the Banach one  is identical. 

\begin{thm}\label{AS thm} Let $E$ be a symmetric quasi-Banach function space on $(0,1)$ having the Kruglov property. We have $Z_E^2(0,\infty)\hookrightarrow E(0,1).$
\end{thm}

%Techniques developed in this paper are particularly effective in treating Question~\ref{quest2} above when $p_1\ne q_2$ (see Theorem \ref{main thm} below).
We provide a negative answer to Question~\ref{quest2} for most cases (in particular, we answer Question~\ref{quest} above), which      extends several results in \cite{LP,AL,AHS,HSSZ}.

\begin{thm}\label{main thm} 
Let $p_1,q_1,p_2,q_2\in (0,\infty).$ If $L_{p_1,q_1}(0,1) \hookrightarrow \mathcal{C}_{p_2,q_2},$ then $p_1=q_1=2$ or $p_1=q_1=q_2>2.$
In particular, $$L_{p_1,q_1}(0,1)\not\hookrightarrow \cC_{p_2,q_2}$$ whenever $p_1\ne q_1$. 
\end{thm}
Theorem \ref{main cor}  above follows from Theorem \ref{main thm}  and \cite[Theorem 6]{AL}.
%If $p_1=q_1=2,$ then $L_{p_1,q_1}(0,1)=L_2(0,1)\hookrightarrow\mathcal{C}_{p_2,q_2}.$ The only case when the Question \ref{quest2} remains open is $p_1=q_1=q_2>2.$ In other words, we do not know whether $L_q(0,1)\hookrightarrow\mathcal{C}_{p,q}$ for $q>2.$ Only the case when $2<p\leq q$ is treated (in the negative) in Proposition 4.1 in \cite{AHS}.
%  {\color{red}Keep it or not? becasue we will remove the last part of this part. A slight modification of the proof  of \cite[Proposition 4.1]{AHS} yields that $L_q(0,1)\not\hookrightarrow \cC_{p,q}$ when $q>2$ and $1<p\le 2$. }

\section{Some properties of $L_{p,q}$-spaces}\label{s:pre}

In this section, we collect some properties of (commutative and noncommutative) $L_{p,q}$-spaces, which are useful in the present paper.

Let $L_0(I)$ be the space of Lebesgue measurable functions either on $I=(0,1)$ or $I=(0,\infty) $ (equipped with Lebesgue measure $m$ and with identification $m$-a.e.) or the set $I=\mathbb{Z}_+$ of all positive integers (equipped with the counting measure). Denote by $S(I)$ the subset of $L_0(I)$ which consists of all functions (or sequences) $f$ such that the distribution function 
$$d_{|f|}(s)= m ( \{t\in I :|f|(t)>s\})$$ 
is finite for some $s>0.$ We denote by $\mu(f)$ the decreasing right-continuous rearrangement of $|f|$ given by 
$$\mu(t;f)=\inf\left\{s \ge 0:\ d_{|f|}(s)\le t \right\},\quad t\ge 0.$$
\begin{definition}\cite{KPS,LT2,BS}
A subspace $E(I)$ of $S(I)$ is called a  quasi-Banach symmetric function space if
\begin{enumerate}[{\rm (i)}]
\item $\left(E(I),\norm{\cdot}_E\right)$ is a quasi-Banach space;
\item  $f\in E(I)$ and $\mu(g)\le\mu(f)$ implies that $g\in E(I)$ and $\norm{g}_E\le \norm{f}_E$.
\end{enumerate}
\end{definition}

\begin{defi}\label{lpq def}\cite[Definition 4.4]{BS}
Let $p,q>0.$ We define
$$L_{p,q}(I)=\left\{f\in S(I):\ \int_0^{\infty}
\mu(t;f)^q dt^{\frac{q}{p}}<\infty\right\}.$$
We equip $L_{p,q}(I)$ with the quasi-norm
$$\norm{f}_{L_{p,q}(I)}:= 
\left( \int_0^{\infty}
\mu(t;f)^q dt^{\frac{q}{p}}\right)^{\frac1q},\quad f\in L_{p,q}(I).$$
If $I=\mathbb{Z}_+$ equipped with the counting measure, then $L_{p,q}(I)$ is denoted by $\ell_{p,q}.$ 
\end{defi}

In what follows, $\{E_{i,j}\}_{i,j\geq0}\subset B(H)$ are the standard matrix units. We denote the uniform operator norm on $B(H)$ by $\left\| \cdot \right \|_{\infty}$ and we denote the identity operator on $H$ by ${\bf 1}.$

A family $\{x_i\}_{i=1}^\infty\subset B(H)$ is called orthogonal from the right (respectively, from  the left) if $x_ix_j^{\ast}=0$ (respectively, $x_i^{\ast}x_j=0$) for all $i,j\in\mathbb{N},$ $i\neq j.$ 
%We caution the reader: a sequence simultaneously orthogonal on the left and on the right is not necessarily orthogonal.

Let $P(B(H))$ be the lattice of all projections in $B(H)$.
For each $x\in B(H)$, the generalized singular value function $t\mapsto \mu(t;x)$, $t>0$, is given by
$$\mu(t;x):=\inf \left\{ \norm{x({\bf 1} - e)}_\infty  : e\in P(B(H)) ,~{\rm Tr}(e)\le t\right\},$$
where ${\rm Tr}$ stands for the standard trace of $B(H)$. 

A $*$-ideal $\cI$ in $B(H)$ is said to be a quasi-Banach operator ideal  if it is equipped with a complete quasi-norm $\left\|\cdot\right\|_{\mathcal{I}}$ such that the mappings $(A,B)\to AB$ and $(A,B)\to BA,$ $A\in\mathcal{I},$ $B\in B(H)$ are continuous from $\cI\times B(H)$ to $\cI$ and from $B(H)\times \cI$ to $\cI$, respectively. Passing to an equivalent quasi-norm if needed, we can always achieve that $$\left\|AB\right\|_{\mathcal{I}},\left\|BA\right\|_{\mathcal{I}}\leq\left\|A\right\|_{\mathcal{I}}\left\|B\right\|_{\infty},\quad A\in\mathcal{I},\quad B\in B(H).$$
A quasi-Banach operator ideal $\cI$ in $B(H)$ is said to be symmetric if the quasi-norm of $\cI$ is monotone with respect to the generalized singular value function, i.e., if $x\in \cI$, $y\in B(H)$ with $\mu(y)\le \mu(x)$, then $y\in \cI$ and $\norm{y}_\cI \le \norm{x}_\cI$.
The quasi-norm on any 
  proper  quasi-Banach operator ideal in $B(H)$ is necessarily symmetric\cite[Theorem 2.6]{BCLS}. 
%\footnote{The case when $\cI\subset K(H)$ follows from \cite[Theorem 2.6]{BCLS}. 
%If $\cI\not\subset K(H)$, then $\cI=B(H)$ (see e.g. \cite[Proposition 2.3]{BCLS} or \cite[Proposition 1]{Garling}).
%In this case, for any $\varepsilon>0$, we have $
 %(\norm{x}_\infty -\varepsilon)\norm{p}_\cI \le \norm{ x e^{|x|}( \norm{x}_\infty -\varepsilon ,\norm{x}_\infty ]}_\cI \le 
%\norm{x}_\cI \le \norm{x}_\infty \norm{{\bf 1}}_\cI$ for any one-dimensional projection $p\in B(H)$.
%}.
%For convenience, throughout this paper, we always consider quasi-Banach operator ideals $\cI$ equipped with symmetric quasi-norms. 
Moreover, we may assume that $\norm{p}_\cI =1$ 
 for any one-dimensional projection  $p$. 

Recall that a quasi-Banach operator ideal $\cI$ in $B(H)$ is said to have the Fatou property if for any increasing positive net $x_j \in \cI$ with $x_j\uparrow x$, we have 
$\norm{x}_\cI =\sup_j \norm{x_j }_\cI$. 
For a net  $\{x_j\}_{j }$ in a quasi-Banach operator ideal $\cI$ 
 having the Fatou property, if $x_j\to_j x$ in the weak operator topology, then $\norm{x}_\cI \le \sup_j\norm{x_j}_\cI$. 
Indeed\footnote{The case for Banach operator ideals in $B(H)$ is well-known, see e.g. \cite[Theorem 4.5.6]{DPS}.}, for any finite-dimensional orthogonal projection $p$ in $B(H)$, we have $p x_j p \to_j  pxp $ in $\cI$.
 Let $x=u|x|$ be the polar decomposition and let 
 $\{p_n\}$
 be a  sequence of finite-dimensional orthogonal projections in $B(H)$ with $p_n\uparrow {\bf 1}$.  
%We have $ |x|^{\frac12} p_n |x|^{\frac12} \uparrow |x|$. 
%Let $ y_n :=  p_n |x|^{\frac12} $. % = u_n \left| p_n |x|^{\frac12} \right| $ be the polar decomposition. 
%We have
%$  |y^*_n| =u_n |y_n| u^*_n $ and $  |y^*_n|^2 =u_n |y_n|^2 u^*_n $.
Hence, 
$\sup_j \norm{x_j}_\cI \ge \norm{u^* x_j }_\cI \ge \norm{ p_n u^* x_ jp_n }_\cI
\to_j \norm{  p_n | x | p_n }_\cI 
%=\norm{y_ny_n^*}_\cI
%=\norm{y_n^*y_n}_\cI
\stackrel{\tiny\mbox{\cite[Prop.3.2.10(ii)]{DPS}}}{=}\norm{|x|^{\frac12} p_n |x|^{\frac12}}_\cI \stackrel{\tiny \mbox{Fatou property}}{\uparrow_n} \norm{|x|}_\cI=\norm{x}_\cI
$.\label{Fatou'}

\begin{defi}
Let $p,q>0.$ The ideal $\mathcal{C}_{p,q}$ in $B(H)$ consists of all $x\in B(H)$ with $\mu(x)\in L_{p,q}(0,\infty).$ Setting $\left\|x\right\|_{\cC_{p,q}}=\left\|\mu(x)\right\|_{L_{p,q}(0,\infty)}$ for $x\in \mathcal{C}_{p,q}$ turns $ \mathcal{C}_{p,q}$ into a quasi-Banach ideal, see e.g. \cite{S14}.
\end{defi}

The following is an important property of $L_{p,q}$-spaces. 
\begin{lem}\label{original CD lemma}\cite{Dilworth,CD88} Let $0<p,q<\infty$.
If $\{x_M\}_{M\geq1}\subset \ell_{p,q}$ is such that $x_M\to0$ in the uniform norm as $M\to \infty $, then there exists a strictly increasing $\mathbb{N}$-valued sequence $\{M_n\}_{n\geq1}$ such that
$$(c_{p,q} ^{(1)})^{-1}\left\|{\bm \alpha}\right\|_{\ell_q}\cdot \inf_M \left\|x_M\right\|_{\ell_{p,q}}\leq \left\|\sum_{n\ge 1} {\bm \alpha}(n) x_{M_n}\right\|_{\ell_{p,q}}\leq c _{p,q}^{(1)}\left\|{\bm \alpha}\right\|_{\ell_q} \cdot \sup_M \left\|x_M\right\|_{\ell_{p,q}},$$
 where $c_{p,q} ^{(1)}$ is a constant depending on $p,q$ only.  
\end{lem}

The proof of the following lemma in the setting of separable 
Banach ideals can be found in \cite[Theorem 2.5]{CDS} or \cite[Lemma 7.1]{DPS2016}.
The proof for the setting of separable quasi-Banach ideals is identical.
\begin{lem}\label{equivalent sequence lemma} Let $\mathcal{I}$ be a separable quasi-Banach ideal in $B(H).$ If $\{x_M\}_{M\geq1}\subset\mathcal{I}$ is a normalised sequence such that $x_M\to0$ in the uniform norm as $M\to \infty $, then there exist a strictly increasing $\mathbb{N}$-valued sequence $\{M_n\}_{n\geq1}$ and a  sequence $\{y_M\}_{M\geq1}$ of mutually disjoint elements such that $y_M\to0$ in the uniform norm 
as $M\to \infty $
and such that 
$$c_{\mathcal{I}}^{-1}
\left\|\sum_{n\geq1}{\bm \alpha}(n)y_{M_n}\right\|_{\mathcal{I}}\leq \left\|\sum_{n\geq1}{\bm \alpha}(n)x_{M_n}\right\|_{\mathcal{I}}\leq c_{\mathcal{I}}
\left\|\sum_{n\geq1}{\bm \alpha}(n)y_{M_n}\right\|_{\mathcal{I}}.$$
Hence, $c_{\mathcal{I}}$ is a constant depending on $\cI$ only. 
\end{lem}
\begin{comment}\begin{proof} By the quasi-triangle inequality, it suffices to prove the assertion for self-adjoint sequences only. Without loss of generality, we may assume that  $\|x_M\|_{\mathcal{I}}=1$ for $M\geq 1.$ Using Theorem 2.5 in \cite{CDS} and passing to a subsequence if needed, we find a sequence $\{p_M\}_{M\geq1}$ of pairwise orthogonal projections such that $(1-p_M)x_M\to0$ in $\mathcal{I}$ as $M\to\infty$ (strictly speaking, Theorem 2.5 in \cite{CDS} is for Banach ideals, however the proof for quasi-Banach ideals is identical). Taking the adjoints, we obtain $x_M(1-p_M)\to0$ in $\mathcal{I}$ as $M\to\infty.$ Consequently, $x_M-p_Mx_Mp_M\to0$ in $\mathcal{I}$ as $M\to\infty.$ Set $y_M=p_Mx_Mp_M$ for $M\geq1.$ The rest of the proof is an elementary exercise.
\end{proof}
\end{comment}

The next result extends Lemma \ref{original CD lemma} to a non-commutative setting, which
is a consequence of  Lemmas \ref{original CD lemma} and   \ref{equivalent sequence lemma} above. %For non-semi-normalised sequences, it is an elementary exercise.

\begin{lem}\label{CD lemma} Let $0<p,q<\infty$. If $\{x_M\}_{M\geq1}\subset  \mathcal{C}_{p,q}$ are such that $x_M\to0$ in the uniform norm as $M\to \infty $, then there exists a strictly increasing $\mathbb{N}$-valued sequence $\{M_n\}_{n\geq1}$ such that
$$ (c_{p,q}^{(2)})^{-1}\left\|{\bm \alpha}\right\|_{\ell_q}
\cdot\inf_{M\geq1}\left\|x_M\right\|_{\mathcal{C}_{p,q}}
    \le \left\|
    \sum_{n\geq1}{\bm \alpha}(n)x_{M_n}\right\|_{\mathcal{C}_{p,q}}
\leq c_{p,q}^{(2)}\left\|{\bm \alpha}\right\|_{\ell_q}\cdot\sup_{M\geq1}
\left\|x_M\right\|_{\mathcal{C}_{p,q}},$$
where $c_{p,q}^{(2)}$ is a constant depending on $p,q$ only. 
\end{lem}

Note that the Kruglov operator $K$ is bounded on $L_{p,q}(0,1),$ $p,q>0$ (see \cite{AS07,AS-uspehi,ASZ} for the definition). This can be seen by combining \cite[Theorem 7]{ASZ} and the Johnson--Schechtman inequality~\cite{JS89}. 
Below, we provide an elementary criterion for a quasi-Banach symmetric function space to have the Kruglov property.

Recall that the lower Boyd index of a quasi-Banach symmetric function space $E(0,1)$ is given by 
$${\bm \alpha}_E:=\lim_{t\to 0}  \frac{\log \norm{\sigma_t}_E}{\log t},$$
where the limit exists by the Fekete lemma\cite[Lemma 3.4.6]{LSZ}. 
Here, $\sigma_t,$ $t>0,$ is the dilation operator, i.e., 
$$(\sigma_tf)(s)=
\begin{cases}
f(\frac{s}{t}),& s\le \min(1,t),\\
0,& t<s\le 1.  
\end{cases}$$
\begin{lem}\label{estimate for Kf} For every measurable function $f$ on $(0,1),$ we have
$$\mu(Kf)\leq \mu(f\otimes (K(\chi_{(0,1)})+1)).$$
Here, $K$ stands for the Kruglov operator and $\chi_{_I}$ is the characteristic function of $I\subset (0,1)$. 
\end{lem}
\begin{proof} By the definition of the 
Kruglov operator\cite{AS05} (see also \cite[(33)]{AS-uspehi}),
we have 
$$Kf=\bigoplus_{n\geq1}\left(\left(\sum_{k=1}^nf_{n,k}\right)\otimes\chi_{(0,\frac1{e\cdot n!})}\right):=\sum_{n\geq1}\left(\left(\sum_{k=1}^nf_{n,k}\right) \chi_{_{E_n} }\right) ,$$
where  $\{f_{n,k}\}_{1\leq k\leq n}$ is a set of  
independent random variables which are 
 all equimeasurable with $f$ and $\{E_n\}_{n\ge 1}$ is a sequence of pairwise disjoint subsets of $(0,1)$  with $m(E_n)=\frac1{e\cdot n!}$, $n\ge 1$. 
By \cite[Corollary 2.3.16]{LSZ}, we have 
$$\mu\left(\sum_{k=1}^nf_{n,k}\right)\leq n \sigma_n\mu(f).$$
Thus, we have 
\begin{align*}
\mu(Kf)&=\mu\left(\bigoplus_{n\geq1}\left(\mu\left(\sum_{k=1}^nf_{n,k}\right)\otimes\chi_{(0,\frac1{e\cdot n!})}\right)\right)\leq \mu\left(\bigoplus_{n\geq1}\left(n\sigma_n\mu(f)\otimes\chi_{(0,\frac1{e\cdot n!})}\right)\right)\\
&=\mu\left(\bigoplus_{n\geq1}\left(n\mu(f)\otimes\chi_{(0,\frac1{e\cdot (n-1)!})}\right)\right)=\mu\left(f\otimes\left(\bigoplus_{n\geq1}n\chi_{(0,\frac1{e\cdot (n-1)!})}\right)\right)\\
&=\mu(f\otimes (K(\chi_{(0,1)})+1)).
\end{align*}
\end{proof}

\begin{lem}\label{boundedness of K} If $E$ is a symmetric quasi-Banach function space on $(0,1)$ such that ${\bm \alpha}_E>0,$ then the mapping $f\to f\otimes K(\chi_{(0,1)})$ is bounded on $E.$
\end{lem}
\begin{proof} By the quasi-triangle inequality of $E$ (with a constant $C_E$), we have 
\begin{eqnarray*}
\left\|f\otimes K(\chi_{(0,1)})\right\|_E
&\stackrel{\mbox{\tiny \cite[(33)]{AS-uspehi}}}{=}&\left\|f\otimes\left(\bigoplus_{n\geq0}n\chi_{(0,\frac1{e\cdot n!})}\right)\right\|_E\\
&\leq&\sum_{n\geq0}nC_E^{n+1}\left\|f\otimes\chi_{(0,\frac1{e\cdot n!})}\right\|_E=
\sum_{n\geq0}nC_E^{n+1}
\left\|\sigma_{\frac1{e\cdot n!}}f\right\|_E.
\end{eqnarray*}
 Since ${\bm \alpha}_E>0,$ it follows from the definition of ${\bm \alpha}_E$ that there exists a constant $c_E\in (0,1) $ depending on $E$ only such that   $$   \frac{\log \norm{\sigma_{\frac1{e\cdot n!}} }_E}{\log  {\frac1{e\cdot n!}}}  \ge  c_E  {\bm \alpha}_E,~n\ge 0.  $$
Since $\frac{1}{e\cdot n!}<1$, it follows that for all $f\in E$, we have 
$$\left\|\sigma_{\frac1{e\cdot n!}} f\right\|_E\leq \left(\frac1{e\cdot n!}\right) ^{ c_E {\bm \alpha}_E}\left\|f\right\|_E,~n\ge 0.$$
Thus, we have 
$$\left\|f\otimes K(\chi_{(0,1)})\right\|_E\leq \sum_{n\geq0}nC_E^{n+1}(e\cdot n!)^{-c_E {\bm \alpha}_E}\cdot \left\|f\right\|_E.$$ 
\end{proof}

\begin{rem}
It follows immediately from the definition of Boyd indices that
both the lower and the upper Boyd indices of $L_{p,q}(0,1)$
equal $p$ (see, e.g., \cite[p.218]{BS}). This together with Lemmas \ref{estimate for Kf} and \ref{boundedness of K} implies that $L_{p,q}(0,1)$ has the Kruglov property. 
\end{rem}

Below, $C_X$ denotes the constant in the quasi-triangle inequality for a quasi-Banach space $X.$ The following quasi-triangle inequalities are valid in every quasi-Banach space $X$:
\begin{equation}\label{qt many summands}
\left\|\sum_{n=1}^Nx_n\right\|_X \le \left\|\sum_{n=1}^{2^{k+1} }x_n\right\|_X \le C_X^{ k+1} \sum_{n=1}^N\left\|x_n\right\|_X
\leq (2N)^{\log_2(C_X)}\sum_{n=1}^N\left\|x_n\right\|_X,   
\end{equation}
where $2^k <  N\le 2^{k+1},~ x_n=0,~n\ge N+1  $ (in particular, we have $ \log_2 (C_X^{k+1})   \le \log_2 (C_X^{\log_2(2N)}) = \log_2(2N)\log_2 (C_X )=\log_2\left( (2N)^
{\log_2 (C_X )} \right) $),
and 
\begin{equation}\label{qt infinitely many summands}
\left\|\sum_{n=1}^{\infty}x_n \right\|_X\leq \sum_{n=1}^{\infty}C_X^n\left\|x_n\right\|_X.
\end{equation}

Let us now recall the upper Boyd index, in the setting of sequence spaces. 
For any $n\in \mathbb{N}$, we define the dilation operator 
$$\sigma_n:\ell_\infty \to \ell_\infty,\quad  \{u(k)\}_{k\geq0} \mapsto \left\{
u(\lfloor \frac{k}{n}\rfloor)\right\}_{k\geq0}.$$
The upper Boyd index associated to a sequence space $J$ is defined by 
$${\bm \beta}_J = \lim_{k\to \infty }\frac{\log \norm{\sigma_k}_{J\to J}}{\log k},$$
where the limit exists by the Fekete lemma. By \eqref{qt many summands}, we have
$$\left\|\sigma_nx\right\|_J\leq (2n)^{\log_2(2C_J)+1}\left\|x\right\|_J,\quad x\in J,\quad n\in\mathbb{N}.$$
 Hence, ${\bm \beta}_J
=\lim_{k\to\infty } \frac{\log \norm{\sigma_k}_{J\to J}}{ \log k}\le 
\lim_{k\to\infty } \frac{\log ( (2k)^{\log_2 (2C_J) +1 })}{ \log k}=\log_2 (2C_J) +1  
<\infty.$

\section{ Non-embedding results: applications of Holub's theorem (1973)}\label{sec:holub}

The theorem below is the main result of this section. Methods used in its proof are certainly known to experts and are scattered across the literature in the setting of Banach spaces, see e.g. \cite{AHS,Maligranda,Dilworth}. However, due to the lack of references for the quasi-Banach setting, we present full proof below. 

\begin{thm}\label{easy thm} Let $p_1,q_1,p_2,q_2>0.$ If $L_{p_1,q_1}(0,1)\hookrightarrow\mathcal{C}_{p_2,q_2},$ then $p_1\geq 2$ and either $p_1=q_1=2$ or $q_1=q_2.$ 
\end{thm}

The following result was claimed   in \cite{Holub}. %The argument in \cite{Holub} is written in a somewhat telegraphic style. 
We refer the reader to Theorem 1 in \cite{Friedman} for a full proof, see also \cite{AL,Arazy81}.

\begin{lem}\label{holub lemma} Every closed subspace in $\mathcal{K}(H)$ is either isomorphic to a Hilbert space or contains a subspace isomorphic to $c_0.$
\end{lem}
%\begin{proof} Let $X$ be a Banach subspace in $\mathcal{K}(H).$ Without loss of generality, $H=l_2.$
%
%For every $n\geq0,$ define the map $\iota_n:X\to \ell_2^{\oplus (2n+2)}$ by setting $$\iota_n(x)=\bigoplus_{0\leq k\leq n}(xe_k\oplus x^{\ast}e_k).$$
%
%{\bf Case 1:} Suppose that none of the maps $\iota_n:X\to \ell_2^{\oplus (2n+2)}$ is an isomorphic embedding. Hence, for every $n\geq0,$ there exists a normalised $x_n\in X$ such that $\|\iota_n(x_n)\|_{\ell_2^{\oplus (2n+2)}}\leq 2^{-n}.$
%
%If $\xi\in\ell_2,$ then
%$$\|x_n\xi\|_{\ell_2}=\|\sum_{k=0}^n\xi(k)x_n(e_k)+x_n(\sum_{k<n}\xi(k)e_k)\|_{\ell_2}\leq$$
%$$\leq\sum_{k=0}^n|\xi(k)|\|x_n(e_k)\|_{\ell_2}+\|x_n(\sum_{k<n}\xi(k)e_k)\|_{\ell_2}\leq$$
%$$\leq\sum_{k=0}^n|\xi(k)|\|x_n(e_k)\|_{\ell_2}+\|\sum_{k<n}\xi(k)e_k\|_{\ell_2}\leq$$
%$$\leq \|\xi\|_{\ell_2}\|\iota_n(x_n)\|_{\ell_2^{\oplus (2n+2)}}+\|\sum_{k<n}\xi(k)e_k\|_{\ell_2}.$$
%The right hand side tends to $0$ as $n\to\infty.$ Since $\xi\in\ell_2$ is arbitrary, it follows that $x_n\to0$ in strong operator topology as $n\to\infty.$ Similarly, $x_n^{\ast}\to0$ in strong operator topology as $n\to\infty.$ Since the sequence $\{x_n\}_{n\geq0}$ is normalised, it follows that there is a subsequence of the sequence $\{x_n\}_{n\geq0}$ which is equivalent to the standard basis in $c_0.$
%
%{\bf Case 2:} Suppose that there exists $n\in\mathbb{Z}_+$ such that the map $\iota_n:X\to \ell_2^{\oplus (2n+2)}$ is an isomorphic embedding. In this case, $X$ is isomorphic to a subspace of a Hilbert space and is, therefore, isomorphic to a Hilbert space.
%\end{proof}

The following lemma is certainly known to experts, see e.g. \cite{AHS,Arazy81} for some special cases. 
However, due to the lack of suitable references, we present a short proof below.
\begin{lem}\label{first easy lemma}
Let $0<p,q,r<\infty.$ If $\ell_r\hookrightarrow \mathcal{C}_{p,q},$ then either $r=2$ or $r=q.$
\end{lem}
\begin{proof} Let $T:\ell_r\to \mathcal{C}_{p,q}$ be an isomorphic embedding. We  have 
\begin{enumerate}
\item[(a)] either $T:\ell_r\to\mathcal{K}(H)$ is an isomorphic embedding; 
\item[(b)] or  there exists a normalised sequence $\{{\bm \alpha}_k\}_{k\geq1}\subset \ell_r$ such that $T({\bm \alpha}_k)\to0$ in $\mathcal{K}(H)$ as $k\to \infty$. 
\end{enumerate}

In case (a), $T(\ell_r)$ is a closed  subspace in $\mathcal{K}(H).$ By Lemma \ref{holub lemma}, we either have $T(\ell_r)\approx\ell_2$ or $c_0\hookrightarrow T(\ell_r).$ Since $T:\ell_r\to\mathcal{K}(H)$ is an isomorphic embedding, it follows that either $\ell_r\approx\ell_2$ or $c_0\hookrightarrow \ell_r.$ By \cite[Proposition  2.9]{KPR} and \cite[Corollary 2.1.6]{AK}, the first option is only possible if $r=2$ and the second option is impossible. Hence, $r=2$ in case (a).

In case (b), 
by Lemma \ref{CD lemma},
 passing to a subsequence if necessary, we conclude that the sequence $\{T({\bm \alpha}_k)\}_{k\geq1}$  (in $\mathcal{C}_{p,q}$) is equivalent to the standard basis in $\ell_q.$ 
Hence, the sequence $\{{\bm \alpha}_k\}_{k\geq1}$ (in $\ell_r$) is equivalent to the standard basis of $\ell_q.$ Hence, $\ell_q\hookrightarrow \ell_r.$  By \cite[Proposition  2.9]{KPR} and \cite[Corollary 2.1.6]{AK}, this is only possible if $r=q.$ Hence, $r=q$ in case (b).
\end{proof}

\begin{lem}\label{type-cotype lemma} 
Let $0<r<\infty$ and $1\leq s\leq 2.$ If $\ell_r\hookrightarrow L_s(0,1),$ then
we have $s\leq r\leq 2.$ Moreover, $c_0\not\hookrightarrow L_s(0,1).$
\end{lem}
\begin{proof} The fact $c_0\not\hookrightarrow L_s(0,1)$ is well-known, see e.g. \cite[Theorem 2.5.6]{MN} (see also \cite[Theorem 5.9.6]{DPS} and \cite[Theorem 6.5]{DPS2016}). We only need to prove the assertion concerning $\ell_r.$ 
 
Let $T:\ell_r\to L_s(0,1)$ is an isomorphic embedding. One can equip $\ell_r$ with the norm ${\bm \alpha}\to\left\|T({\bm \alpha})\right\|_s.$ Since $T$ is an isomorphic embedding, it follows that the natural quasi-norm on $\ell_r$ is equivalent to this norm. Using Lemma 2.7 in \cite{KPR}, we conclude that $r\geq 1.$ 
Now, the assertion follows from  \cite[Theorem 6.4.19]{AK}.
\end{proof}

\begin{fact}\label{rstable fact} Let $0<r<2.$ It is a standard fact (see e.g. Theorem 5 on p.181 in \cite{GnedenkoKolmogorov}) that $r$-stable random variables belong  to $L_{r,\infty}(0,1)$ (the weak $L_r$-space\cite{BS}). Let $\{f_k\}_{k\geq1}$ be a sequence of independent $r$-stable random variables. Since
$$\mu\left(\sum_{k\geq1}{\bm \alpha}(k)f_k\right)=\left\|{\bm \alpha}\right\|_r\mu(f_1)$$
(see e.g. \cite[p. 94]{Wojtaszczyk}),
it follows that the linear span of the sequence $\{f_k\}_{k\geq1}$ in $E(0,1)$ is an isometric copy of $\ell_r$ provided that $L_{r,\infty}(0,1)\subset E(0,1)$, see also \cite{GM,Wojtaszczyk}.
\end{fact}

\begin{rem} The conclusion of Lemma \ref{type-cotype lemma} is true for $s\in(0,2]$ (see \cite[Theorem 2.3]{Maligranda}). Conversely, it follows from Fact \ref{rstable fact} that $\ell_r\subset L_s(0,1)$ when $0<s<r<2.$ It is clear that $\ell_2,\ell_s \hookrightarrow L_s(0,1)$ for $0<s<2.$

If $s\geq 2$ and $r>0$ are such that $\ell_r\hookrightarrow L_s(0,1),$ then either $r=2$ or $r=s$ (see \cite[Theorem 6.4.19]{AK}). 
It is well-known that $\ell_2\hookrightarrow L_2(0,1)$ (by the Khintchine inequality or Theorem \ref{AS thm} above) and $\ell_s\hookrightarrow L_s(0,1).$

The fact that $c_0\not\hookrightarrow L_s(0,1)$ when $s\ge 1$ follows from   \cite[Theorem 2.5.6]{MN} (the general case when $0<s<\infty$ follows from \cite[Theorem 2.3]{Maligranda}). 
\end{rem}

\begin{lem}\label{second easy lemma}
Let $0<p_1,p_2,q_1,q_2<\infty.$ If $L_{p_1,q_1}(0,1)\hookrightarrow\mathcal{C}_{p_2,q_2},$ then $p_1\geq 2.$
\end{lem}
\begin{proof} Assume by contradiction  that $0<p_1<2.$ Fix $r\neq 2$ such that $r\ne q_2$ and $p_1<r<2.$ We have $L_{r,\infty}(0,1)\subset L_{p_1,q_1}(0,1)$ (see e.g. \cite[p.217]{BS}). According to Fact \ref{rstable fact}, $\ell_r\hookrightarrow L_{p_1,q_1}(0,1).$ Hence, $\ell_r\hookrightarrow \mathcal{C}_{p_2,q_2}.$ By Lemma \ref{first easy lemma}, we have either $r=2$ or $r=q_2,$ which is excluded by the choice of $r.$ This contradiction shows that   $ p_1 \ge 2.$
\end{proof}

The following lemma in the Banach setting (i.e. $p>1$ and $q\ge 1$) is well-known \cite[Theorem 11]{Dilworth}. Due to the lack of suitable references in the quasi-Banach setting we include a short proof below.
\begin{lem}\label{fifth easy lemma}
Suppose that $2\leq p<\infty$ and $0<q,r<\infty.$ If $\ell_r\hookrightarrow L_{p,q}(0,1),$ then either $r=q$ or $r=2.$ Moreover, $c_0\not\hookrightarrow L_{p,q}(0,1).$
\end{lem}
\begin{proof} We only prove the assertion concerning $\ell_r.$ The proof in the case of $c_0$ is nearly identical.

Fix an isomorphic embedding $T:\ell_r\to L_{p,q}(0,1).$ We  have 
\begin{enumerate}
\item[(a).] either for every $s<p,$ $T:\ell_r\to L_s(0,1)$ is an isomorphic embedding;
\item[(b).]
 or  there exists $s\in(0,p)$ and a normalised sequence $\{{\bm \alpha}_k\}_{k\geq1}\subset\ell_r$ such that $T({\bm \alpha}_k)\to0$ in $L_s(0,1).$
\end{enumerate}

In case (a), for $p>2,$ we have $\ell_r\hookrightarrow  L_2(0,1).$ By Lemma \ref{type-cotype lemma}, we have $r=2.$

In case (a), for $p=2,$ we have $\ell_r\hookrightarrow   L_s(0,1)$ for every $1\leq s<2.$ By Lemma~\ref{type-cotype lemma}, this means $s\leq r\leq 2$ for every $1\leq s<2.$ Hence, $r=2.$

In case (b), we have $T({\bm \alpha}_k)\to0$ in measure as $k\to \infty $. 
By Lemma 2.1 in \cite{CD88}, the sequence $\{T({\bm \alpha}_k)\}_{k\geq1}$ in $L_{p,q}(0,1)$ contains a subsequence equivalent to the standard basis in $\ell_q.$ Hence, the sequence $\{{\bm \alpha}_k\}_{k\geq1}$ in $\ell_r$ contains a subsequence equivalent to the standard basis in $\ell_q.$ In other words, $\ell_q\hookrightarrow\ell_r.$ According to \cite[Proposition  2.9]{KPR} and \cite[Corollary 2.1.6]{AK}, we have $r=q.$
\end{proof}

\begin{lem}\label{third easy lemma}Let $0< p_1,p_2, q_1,q_2<\infty.$ If $L_{p_1,q_1}(0,1)\hookrightarrow\mathcal{C}_{p_2,q_2},$ then either $q_1=q_2$ or $q_1=2.$
\end{lem}
\begin{proof} By \cite[Proposition 1]{Dilworth}, $\ell_{q_1}\hookrightarrow L_{p_1,q_1}(0,1).$ Thus, $\ell_{q_1}\hookrightarrow\mathcal{C}_{p_2,q_2}.$ The assertion follows now from Lemma \ref{first easy lemma}. 
\end{proof}

The following fact follows from a combination of  
\cite[Theorem 7.4.1]{AK} and known results concerning  types/cotypes of $L_{p,q}$-spaces\cite{Creekmore, Maligranda}, see also \cite[Corollary 2.e.8]{LT2} and \cite[Corollary 1.7]{JMST}. 
\begin{fact}\label{factLpq} Let $0<p,q<\infty.$ If $L_{p,q}(0,1)$ is isomorphic to $L_2(0,1)$, then $p=q=2.$  
\end{fact}

\begin{lem}\label{fourth easy lemma}
Let $0<p_1,p_2,q_1,q_2<\infty.$ If $L_{p_1,q_1}(0,1)\hookrightarrow\mathcal{C}_{p_2,q_2},$ then either $p_1=q_1=2$ or $q_2=q_1$ or $q_2=2.$
\end{lem}
\begin{proof} By Lemma \ref{second easy lemma}, we have  $p_1\geq 2.$

Let   $T:L_{p_1,q_1}(0,1)\to\mathcal{C}_{p_2,q_2}$
be  an isomorphic embedding.
 We consider the natural continuous embedding from $\mathcal{C}_{p_2,q_2}$ into $\mathcal{K}(H)$. Then, 
\begin{enumerate}
\item[(a)] either $T$ generates isomorphic embedding from  $L_{p_1,q_1}(0,1)$ into $\mathcal{K}(H)$;
\item[(b)] or there exists a normalised sequence $\{f_k\}_{k\geq1}\subset L_{p_1,q_1}(0,1)$ such that $T(f_k)\to0$ in $\mathcal{K}(H)$ as $k\to\infty.$ 
\end{enumerate}

In case (a), $T(L_{p_1,q_1}(0,1))$ is a (Banach) subspace in $\mathcal{K}(H).$ By Lemma \ref{holub lemma}, we either have $T(L_{p_1,q_1}(0,1))\approx\ell_2$ or $c_0\hookrightarrow T(L_{p_1,q_1}(0,1)).$ Since $T:L_{p_1,q_1}(0,1)\to\mathcal{K}(H)$ is an isomorphic embedding, it follows that either $L_{p_1,q_1}(0,1)\approx\ell_2$ or $c_0\hookrightarrow L_{p_1,q_1}(0,1).$ By Lemma \ref{fifth easy lemma}, the second option is impossible. Hence, $L_{p_1,q_1}(0,1)\approx\ell_2.$ Using Fact \ref{factLpq}, we conclude that $p_1=q_1=2$ in case (a).
	
In case (b), we have $T(f_k)\to0$ in the uniform norm as $k\to \infty.$  Since the sequence $\left\{f_k\right\}_{k\geq1}$ is normalised in $L_{p_1,q_1}(0,1),$ it follows that the sequence $\left\{T(f_k)\right\}_{k\geq1}$ is semi-normalised in $\mathcal{C}_{p_2,q_2}$ (i.e., $\infty >\sup_{k\ge 1}\norm{T(f_k)}_{\mathcal{C}_{p_2,q_2}} \ge \inf_{k\ge 1}\norm{T(f_k)}_{\mathcal{C}_{p_2,q_2}}>0$).
By Lemma~\ref{CD lemma},
 passing to a subsequence if necessary, we infer that the sequence $\left\{T(f_k)\right\}_{k\geq1}\subset\mathcal{C}_{p_2,q_2}$ is equivalent to the unit basis in $\ell_{q_2}.$ Hence, the sequence $\left\{f_k\right\}_{k\geq1}\subset L_{p_1,q_1}(0,1)$ is equivalent to the unit basis in $\ell_{q_2}.$ In particular, $\ell_{q_2}\hookrightarrow L_{p_1,q_1}(0,1).$ By Lemma \ref{fifth easy lemma}, we either have $q_2=q_1$ or $q_2=2$ in case (b).
\end{proof}

\begin{proof}[Proof of Theorem \ref{easy thm}]  By Lemma \ref{second easy lemma}, we have  $p_1\geq 2.$ By Lemma \ref{fourth easy lemma}, we have either $p_1=q_1=2$ or $q_2=q_1$ or $q_2=2.$ On the other hand, Lemma \ref{third easy lemma} asserts that either $q_1=2$ or $q_1=q_2.$ Combining the last two assertions, we complete the proof.
\end{proof}

\section{Estimate of the limit points}

The following theorem is the main result of this section. This should be compared with the estimates in \cite[Lemma 4]{AL} and \cite[Proposition 2.9]{AHS}.
\begin{thm}\label{limit point theorem} Let $\Theta\in\mathbb{R}_+$ and $n\in\mathbb{N}.$ Let $\left(\mathcal{I},\left\|\cdot\right\|_{\mathcal{I}}\right)$ be a quasi-Banach ideal in $B(H)$ having the Fatou property. Let $\{x_j\}_{j\geq 1}\subset \mathcal{I}$ be a bounded sequence such that $x_j\to0$ in the weak operator topology and that 
$$\left\|\sum_{j\in A}x_j\right\|_{\mathcal{I}}\leq \Theta|A|^{\frac12},\quad A\subset\mathbb{Z}_+,\quad |A|\geq n.$$
If $x$ is a limit point of the sequence $\{x_j^{\ast}x_j\}_{j\geq1}$ in the weak operator topology, then $\left\|x^{\frac12}\right\|_{\mathcal{I}}\leq C_{\mathcal{I}}^2\Theta.$
\end{thm}

%For $f,g\in (L_1+L_\infty)(I),$ we say that $g$ is submajorized by $f$ in the sense of  Hardy--Littlewood--Polya (written $g\prec\prec f$) if $$\int_0^t\mu(s,g)ds\leq\int_0^t\mu(s,f)ds,\quad t>0.$$ 

The following lemma is an extension of the main result in \cite{Ricard}.
For its full proof, we refer to 
  \cite[Theorem 3.1]{HNSZ} (see also  \cite[Theorem 6.3]{HSZ} and \cite{HS20}).
\begin{lem}\label{bks lemma} 
Let $\left(\mathcal{J},\left\|\cdot\right\|_{\mathcal{J}}\right)$ be a quasi-Banach ideal in $B(H).$ If $X,Y\in\mathcal{J}$ are self-adjoint, then for any $ p>   1$, we have 
$$\left\||X|^{\frac1p}-|Y|^{\frac1p}\right\|_{\mathcal{J}}\leq c_{\mathcal{J},p}'\left\||X-Y|^{\frac1p}\right\|_{\mathcal{J}}.$$
Here, $c'_{\cJ,p}$ is a constant depending on $\cJ$ and $p$ only. 
\end{lem}
%\begin{proof} Let the ideal $\left(\mathcal{J},\left\|\cdot\right\|_{\mathcal{J}}\right)$ correspond to a quasi-Banach symmetric sequence space $J$\cite{S14}. Since the upper Boyd index ${\bm \beta}_J$ is finite, setting $p_J=\frac1{2{\bm \beta}_J}$, it follows from Proposition 6.7 in \cite{CSZ} (see also \cite[Lemma 3.5]{HNSZ}) that there exists a constant $c_J$ such that, for every $x\in J$ and $y\in \ell_{\infty}$ with $|y|^{p_J}\prec\prec |x|^{p_J},$ we have $y\in J$ and $\left\|y\right\|_J\leq c_J\left\|x\right\|_J.$ 
%Applying Theorem 6.1 in \cite{HSZ}, we obtain 
%$$\Big||X|^{\frac1p}-|Y|^{\frac1p}\Big|^{p_J}\prec\prec c_{p_J,p}|X-Y|^{\frac{p_J}{p}}.$$
%Hence,
%$$\||X|^{\frac1p}-|Y|^{\frac1p}\|_{\mathcal{J}}\leq c_Jc_{p_J,p }\||X-Y|^{\frac1p}\|_{\mathcal{J}}.$$
%\end{proof}
The following lemma follows from Lemma \ref{bks lemma} via the same argument used in \cite[Corollary 7.5]{HSZ}.
For the sake of completeness, we present a short (and slightly different) proof below. 
\begin{lem}\label{abs holder estimate} Let $(\mathcal{I},\left\|\cdot\right\|_{\mathcal{I}})$ be a quasi-Banach ideal in $B(H).$ We have
$$\left\||X|-|Y|\right\|_{\mathcal{I}}\leq c_{\mathcal{I}}''\cdot\max\left\{\left\|X\right\|_{\mathcal{I}}^{\frac12},\left\|Y\right\|_{\mathcal{I}}^{\frac12}\right\}\cdot \|X-Y\|_{\mathcal{I}}^{\frac12}.$$
Here, $c_\cI''$ is a constant depending on $\cI$ only. 
\end{lem}
\begin{proof} Let $\mathcal{J}$ be the $2$-convexification of $\mathcal{I}$ (see \cite{ALin}, \cite[p.97]{DDS14}), that is, 
$$\mathcal{J}=\{x\in B(H): |x|^2 \in \mathcal{I}\}, ~\norm{x}_{\mathcal{J}} =\norm{|x|^2}_{\mathcal{I}}^{\frac12},$$
which is a quasi-Banach ideal\cite[Theorem 3.1]{DDS14}. 
We have a H\"older inequality (see e.g. \cite[Lemma 3.8]{HNSZ} or   \cite[Theorem 1]{Sukochev16})
$$\left\|AB\right\|_{\mathcal{I}}\leq  c_{\mathcal{I}}\left\|A\right\|_{\mathcal{J}}\left\|B\right\|_{\mathcal{J}},~A,B\in \cJ.$$

Suppose first $X$ and $Y$ are self-adjoint. We write
$$|X|-|Y|=|X|^{\frac12}(|X|^{\frac12}-|Y|^{\frac12})+(|X|^{\frac12}-|Y|^{\frac12})|Y|^{\frac12}.$$
The assertion follows now from quasi-triangle inequality, the H\"older inequality stated above and Lemma \ref{bks lemma}.
Indeed, 
\begin{align*}
\norm{|X|-|Y|}_\cI
&\le C_\cI \left(  \norm{|X|^{\frac12}(|X|^{\frac12}-|Y|^{\frac12})}_\cI+\norm{ (|X|^{\frac12}-|Y|^{\frac12})|Y|^{\frac12}}_\cI \right)\\
&\le C_\cI c_\cI \left(
\norm{ |X|^{\frac12}-|Y|^{\frac12} }_\cJ \norm{|X|^{\frac12}}_\cJ +\norm{  |X|^{\frac12}-|Y|^{\frac12} }_\cJ\norm{|Y|^{\frac12}}_\cJ \right)\\
&\le   C_\cI c_\cI c'_{\cJ,2}\cdot \max\left\{\left\|X\right\|_{\mathcal{I}}^{\frac12},\left\|Y\right\|_{\mathcal{I}}^{\frac12}\right\} \cdot 
\norm{ |X - Y|^{\frac12} }_\cJ \\
&=   C_\cI c_\cI c'_{\cJ,2}\cdot \max\left\{\left\|X\right\|_{\mathcal{I}}^{\frac12},\left\|Y\right\|_{\mathcal{I}}^{\frac12}\right\} \cdot 
\norm{ X - Y  }_\cI^{\frac12} .
\end{align*}

Consider now the general case. Set 
$$X_0=X\otimes E_{2,1}+X^{\ast}\otimes E_{1,2},\quad Y_0=Y\otimes E_{2,1}+Y^{\ast}\otimes E_{1,2}.$$
It is immediate that $X_0=X_0^{\ast},$ $Y_0=Y_0^{\ast},$
$$|X_0|=|X|\otimes E_{1,1}+|X^{\ast}|\otimes E_{2,2},\quad |Y_0|=|Y|\otimes E_{1,1}+|Y^{\ast}|\otimes E_{2,2}.$$
Hence,
$$\left\||X|-|Y|\right\|_{\mathcal{I}}=\left\|(|X_0|-|Y_0|)\cdot (1\otimes E_{1,1})\right\|_{\mathcal{I}}\leq\left\||X_0|-|Y_0|\right\|_{\mathcal{I}}$$
and 
$$\left\|X_0-Y_0\right\|_{\mathcal{I}}\leq 2C_{\mathcal{I}}\left\|X-Y\right\|_{\mathcal{I}}.$$
Hence, the inequality in the general case follows from the one in the self-adjoint case.
Indeed, \begin{align*}\left\||X|-|Y|\right\|_{\mathcal{I}}&\le \left\||X_0|-|Y_0|\right\|_{\mathcal{I}}\\
&\leq  C_\cI c_\cI c'_{\cJ,2}\cdot \max\left\{\left\|X_0\right\|_{\mathcal{I}}^{\frac12},\left\|Y_0\right\|_{\mathcal{I}}^{\frac12}\right\}\cdot \|X_0-Y_0\|_{\mathcal{I}}^{\frac12}\\
& \le 
 C_\cI c_\cI c'_{\cJ,2}\cdot  \max\left\{ \left ( 2C_\cI\left\|X\right\|_{\mathcal{I}}\right)^{\frac12}, \left(2 C_\cI \left\|Y\right\|_{\mathcal{I}}\right)^{\frac12}\right\}\cdot \left( 2C_{\mathcal{I}}\left\|X-Y\right\|_{\mathcal{I}} \right) ^{\frac12}
. 
\end{align*}
\end{proof} 

\begin{lem}\label{easy limit point lemma} Let $\Theta\in\mathbb{R}_+$ and $n\in\mathbb{N}.$ Let $\left(\mathcal{I},\left\|\cdot\right\|_{\mathcal{I}}\right)$ be a quasi-Banach ideal in $B(H).$ Let the sequence $\{u_j\}_{j\geq 1}\subset\mathcal{I}$ be disjointly supported   from the left (i.e., $u_k^* u_j=0$ for any $k\ne j$) and such that
\begin{align}\label{easy assumption Theta}
\left\|\sum_{j\in A}u_j\right\|_{\mathcal{I}}\leq \Theta|A|^{\frac12},\quad A\subset\mathbb{Z}_+,\quad |A|\geq n.
\end{align}
If $u$ is a limit point (with respect to the quasi-norm $\left\|\cdot\right\|_\mathcal{I}$) of the sequence $\{|u_j|\}_{j\geq1},$ then $\left\|u\right\|_{\mathcal{I}}\leq  C_{\mathcal{I}}\Theta.$ 
\end{lem}
\begin{proof} Without loss of generality, 
we may assume that 
$\left\|u\right\|_{\mathcal{I}}\leq 1$ and $\left\|u_j\right\|_{\mathcal{I}}\leq 1$ for every $j\geq 1.$ Passing to a subsequence if needed, we may assume without loss of generality that $|u_j|\to u$ in $\mathcal{I}$ and that
\begin{equation}\label{elpl eq0}
\left\||u_j|-u\right\|_{\mathcal{I}}\leq j^{-1},\quad j\geq1.
\end{equation}

Observe that 
$$\left|\sum_{j\in A}u\otimes E_{j,1}\right|=|A|^{\frac12}u\otimes E_{1,1}.$$
Therefore, we  have
\begin{eqnarray*}
|A|^{\frac12}\left\|u\right\|_{\mathcal{I}}&=&\left\|\sum_{j\in A}u\otimes E_{j,1}\right\|_{\mathcal{I}}\\
&=&\left\|
\sum_{j\in A}|u_j|\otimes E_{j,1}+\sum_{j\in A}(u-|u_j|)\otimes E_{j,1}\right\|_{\mathcal{I}}\\
&\leq& 
C_{\mathcal{I}}\left(\left\|\sum_{j\in A}|u_j|\otimes E_{j,1}\right\|_{\mathcal{I}}+
\left\|\sum_{j\in A}(u-|u_j|)\otimes E_{j,1}\right\|_{\mathcal{I}}\right)\\
&\stackrel{\eqref{qt many summands}}{\leq}& C_{\mathcal{I}}\left(
\left\|\sum_{j\in A}|u_j|\otimes E_{j,1}\right\|_{\mathcal{I}}+(2|A|)^{\log_2(C_{\mathcal{I}})}\sum_{j\in A}
\left\|u-|u_j|\right\|_{\mathcal{I}}\right)\\
&\stackrel{\eqref{elpl eq0}}{\leq}& C_{\mathcal{I}}\left(\left\|\sum_{j\in A}|u_j|\otimes E_{j,1}\right\|_{\mathcal{I}}+(2|A|)^{\log_2(C_{\mathcal{I}})+1}\cdot\frac1{\min(A)}\right).
\end{eqnarray*}
Since the sequence $\left\{u_j\right\}_{j\geq 1}\subset\mathcal{I}$ is disjointly supported from   the left, it follows that 
\begin{align*}
\left| \sum_{j\in A}|u_j|\otimes E_{j,1}\right|&=\left(\left(  \sum_{j\in A}|u_j|\otimes E_{j,1} \right)^* \left(  \sum_{j\in A}|u_j|\otimes E_{j,1} \right) \right)^{\frac12}\\
&=\left(\sum_{j\in A}|u_j|^2\right)^{\frac12}\otimes E_{1,1}\\
&=\left(\left(  \sum_{j\in A}u_j\right) ^* \left(  \sum_{j\in A}u_j\right) \right)^{\frac12}\otimes E_{1,1}\\
&=\left(\left|\sum_{j\in A}u_j\right|^2 \right)^{\frac12}\otimes E_{1,1}=
\left|\sum_{j\in A}u_j\right|\otimes E_{1,1}.
\end{align*}
Thus, we have 
\begin{eqnarray*}
|A|^{\frac12}\left\|u\right\|_{\mathcal{I}}
&\leq&
 C_{\mathcal{I}}\left(
 \left\|\sum_{j\in A}u_j\right\|_{\mathcal{I}}+(2|A|)^{\log_2(C_{\mathcal{I}})+1}\cdot\frac1{\min(A)}\right)\\
&
\stackrel{\eqref{easy assumption Theta}}{\leq}& C_{\mathcal{I}}\left(\Theta |A|^{\frac12}+(2|A|)^{\log_2(C_{\mathcal{I}})+1}\cdot\frac1{\min(A)}\right).
\end{eqnarray*}  
Taking $A=\{m,m+1,\cdots,m+n-1\}$ and dividing the above inequality by $|A|^{\frac12},$ we obtain
$$\|u\|_{\mathcal{I}}\leq C_{\mathcal{I}}\Big(\Theta+ (2n)^{\log_2(C_{\mathcal{I}})+1}\cdot \frac1{mn^{\frac12}}\Big),\quad m,n\in\mathbb{N}.$$
Since $m\in\mathbb{N}$ can be chosen arbitrarily large, the assertion follows.
\end{proof}

\begin{proof}[Proof of Theorem \ref{limit point theorem}] 
Without loss of generality, 
passing to a subsequence if needed, we may assume   that $x_j^{\ast}x_j\to x$ in the weak operator topology as $j\to\infty.$

Fix a finite rank  projection $p\in B(H)$ of the form ${\rm Diag(1,\cdots,1,0\cdots)}$ and denote for brevity $u=|x^{\frac12}p|.$ We have $px_j^{\ast}x_jp\to px^2p$ in the weak operator topology as $j\to \infty $. 
For matrices of a fixed size, convergence in the weak operator topology yields the convergence with respect to 
any quasi-Banach ideal in $B(H)$ (in particular, with respect to 
 $\mathcal{I}^2$).
 Hence, $px_j^{\ast}x_jp\to pxp$ as $j\to \infty $ in $\mathcal{I}^2$ of $\cI$.
 In other words, $|x_jp|^2\to u^2$ in $\mathcal{I}^2$ as $j\to \infty$. By Lemma \ref{bks lemma}, $|x_jp|\to u$ in $\mathcal{I}$  as $j\to \infty $. 

If the sequence $\{w_j\}_{j\geq1}\subset\mathcal{K}(H)$ is such that $w_j^{\ast}\to0$ in strong operator topology, then, passing to a subsequence if needed, we write $w_j=u_j+v_j,$ where the sequence $\{u_j\}_{j\geq1}$ is disjointly supported from  the left and  $v_j\to0$ in the uniform norm as $j\to\infty.$

Fix $\epsilon\in(0,1).$ Set $w_j=x_jp$ for $j\geq1.$ Since $p$ is a finite rank  projection and since $x_j\to0$ in the weak operator topology, it follows that $w_j^{\ast}\to0$ in strong operator topology. By the preceding paragraph, we write
$x_jp=u_j+v_j,$ where the sequence $\left\{u_j\right\}_{j\geq1}$ is disjointly supported from  the left and 
$v_j\to0$ in the uniform norm as $j\to\infty.$ Without loss of generality,
we may write 
 $u_j=u_jp$ and $v_j=v_jp$ for $j\geq1.$ Hence, $\left\|v_j\right\|_{\mathcal{I}}\leq\left\|p\right\|_{\mathcal{I}}\left\|v_j\right\|_{\infty}\to0$ as $j\to\infty.$ Passing to a subsequence again, we may assume that
$$x_jp=u_j+v_j,$$ 
where the sequence $\{u_j\}_{j\geq1}$ is disjoint from  the left and where $\left\|v_j\right\|_{\mathcal{I}}\leq (2C_{\mathcal{I}})^{-j}\epsilon$ for every $j\geq 1.$ We have
\begin{eqnarray*}
\left\|\sum_{j\in A}u_j\right\|_{\mathcal{I}}=\left\|\sum_{j\in A}x_jp-\sum_{j\in A}v_j\right\|_{\mathcal{I}}&\stackrel{\eqref{qt infinitely many summands}}{\leq}& C_{\mathcal{I}}\|\sum_{j\in A}x_jp\|_{\mathcal{I}}+C_{\mathcal{I}}\sum_{j\in A}C_{\mathcal{I}}^j\left\|v_j\right\|_{\mathcal{I}}\\
&\leq& C_{\mathcal{I}}\Theta |A|^{\frac12}+\sum_{j\in A}2^{-j}\epsilon\\
&\leq&  \Theta C_{\mathcal{I}}|A|^{\frac12}+\epsilon\leq (C_{\mathcal{I}}\Theta+\epsilon)|A|^{\frac12}
\end{eqnarray*}
   for all $  A\subset\mathbb{Z}_+$ with $ |A|\geq n.$

Since $x_jp-u_j\to0$ in $\mathcal{I}$ as $j\to\infty,$ it follows from Lemma \ref{abs holder estimate} that $|x_jp|-|u_j|\to0$ in $\mathcal{I}$ as $j\to\infty.$ Since also $|x_jp|\to u$ in $\mathcal{I}$ as $j\to\infty,$ it follows that $|u_j|\to u$ in $\mathcal{I}$ as $j\to\infty.$ Applying Lemma \ref{easy limit point lemma} to the sequence $\{u_j\}_{j\geq 1},$ we obtain
$$\left\|u\right\|_{\mathcal{I}}\leq C_{\mathcal{I}}(C_{\mathcal{I}}\Theta+\epsilon).$$
Since $\epsilon$ is arbitrarily small, it follows that $\left\|u\right\|_{\mathcal{I}}\leq C_{\mathcal{I}}^2\Theta.$ Recall that $u=|x^{\frac12}p|.$ Thus,
$$\left\|x^{\frac12}p\right\|_{\mathcal{I}}\leq C_{\mathcal{I}}^2\Theta$$
for an   arbitrary finite rank   projection $p$  of the form ${\rm Diag(1,\cdots,1,0\cdots)}$ 
 in $B(H)$. The assertion follows now from the Fatou property of $\cI$ 
(see page \pageref{Fatou'}).
\end{proof}

\section{A decomposition theorem of sequences in quasi-Banach ideals}

The following decomposition theorem is the main result of this section.
Such a decomposition can be traced back to   \cite{AL} and \cite{Arazy81}.
In fact, in the setting of Banach operator ideals in $B(H)$, Arazy's decomposition theorem \cite[Corollary 2.8]{Arazy81}  is a direct consequence of the following theorem,
%which is useful in the study of the structure of operator ideals in $B(H)$. It should be compared with \cite[Corollary 2.8]{Arazy81}, which was established in the Banach setting (note that Arazy's result holds for sequences tending to $0$ in the weak operator topology \cite[p.307]{Arazy81}).
%Below, 
while our proof is more transparent and straightforward than that in \cite{Arazy81}. 
\begin{thm}\label{fedor decomposition theorem} Let $\left(\mathcal{I}, \left\|\cdot\right\|_{\mathcal{I}}\right)$ be a separable quasi-Banach ideal in $B(H)$ possessing the Fatou property. Let $\{x_j\}_{j\geq1}\subset \mathcal{I}$ be a bounded sequence such that $x_j\to0$ in the  weak operator topology. There exist
\begin{enumerate}[{\rm (i)}]
\item a limit point $y$ of the sequence $\{x_j^{\ast}x_j\}_{j\geq1}$ in the weak operator topology;
\item a limit point $z$ of the sequence $\{x_jx_j^{\ast}\}_{j\geq1}$ in the weak operator topology;
\item a strictly increasing $\mathbb{N}$-valued sequence $\{j_m\}_{m\geq1};$
\item a sequence $\{p_m\}_{m\geq1}$ of pairwise orthogonal   (finite-dimensional) projections  of the form ${\rm Diag(\cdots, 0,1\cdots,1,0\cdots)}$ in $B(H)$;
\item a decomposition
$$x_{j_m}=a_m+b_m+c_m+d_m,\quad m\geq1$$
\end{enumerate}
such that
\begin{enumerate}[{\rm (a)}]
\item\label{deca} $a_m=p_mx_{j_m}p_m$ for every $m\geq1;$
\item\label{decb} $b_m=p_mb_m$ and $c_m=c_mp_m$ for every $m\geq1;$
\item\label{decc} $|b_m|\to y^{\frac12},$ $|c_m^{\ast}|\to z^{\frac12}$ and $d_m\to0$ in $\mathcal{I}$ as $m\to\infty.$
\end{enumerate}
\end{thm}

The positions of $a_m,b_m,c_m$ in the above theorem are demonstrated as follows. 
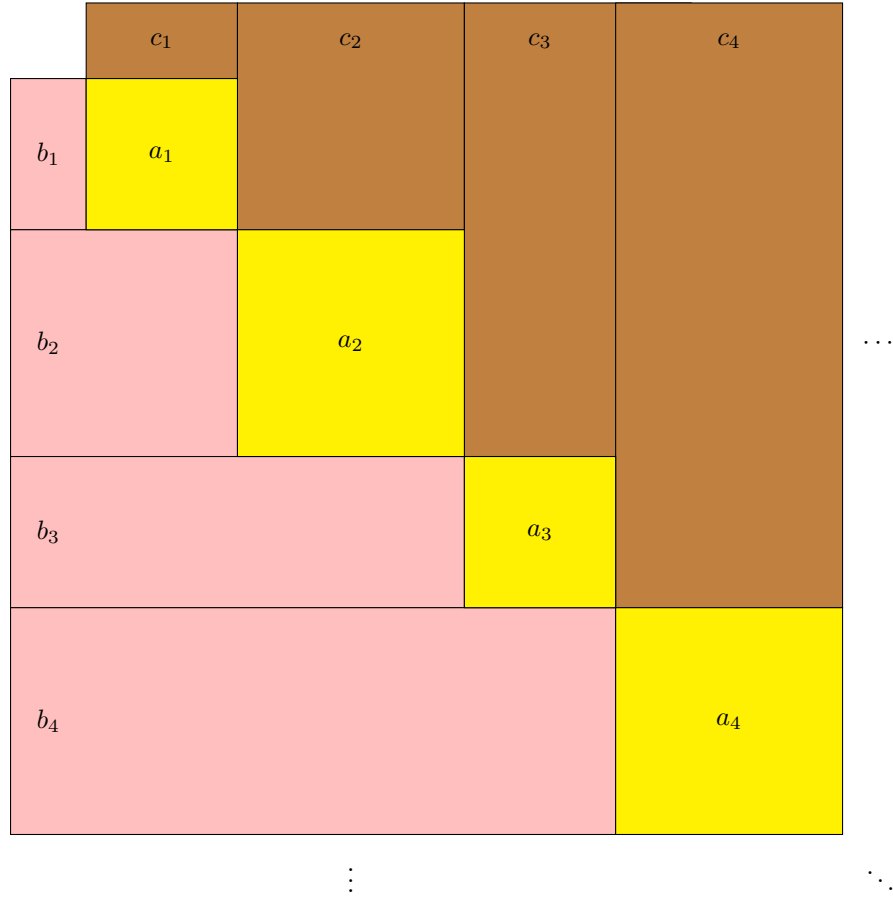
\begin{center}
\begin{tikzpicture}
    \draw[fill=brown] (0,0) rectangle (2,1);
    \node at (1, 0.5) {$c_1$};
    \draw[fill=brown] (2,1) rectangle (5,-2);
    \node at (3.5, 0.5) {$c_2$};
    \draw[fill=brown] (5,1) rectangle (8,-5);
    \node at (6, 0.5) {$c_3$};
    \draw[fill=brown] (7,1) rectangle (10,-7);
    \node at (8.5, 0.5) {$c_4$};
\node at (10.5, -3.5) {$\cdots$};

 \draw[fill=pink] (-1,0) rectangle (0,-2);
    \node at (-0.5, -1) {$b_1$};
 \draw[fill=pink] (-1,-2) rectangle (2,-5);
    \node at (-0.5, -3.5) {$b_2$};
 \draw[fill=pink] (-1,-5) rectangle (5,-7);
    \node at (-0.5, -6) {$b_3$};
 \draw[fill=pink] (-1,-7) rectangle (7,-10);
    \node at (-0.5, -8.5) {$b_4$};
\node at (3.5, -10.5 ) {$\vdots$};

\draw[fill=yellow] (0,0) rectangle (2,-2);
    \node at (1, -1) {$a_1$};
\draw[fill=yellow] (2,-2) rectangle (5,-5);
    \node at (3.5, -3.5) {$a_2$};
\draw[fill=yellow] (5,-5) rectangle (7,-7 );
    \node at (6, -6) {$a_3$};
\draw[fill=yellow] (7,-7) rectangle (10,-10);
    \node at (8.5, -8.5) {$a_4$};

\node at (10.5, -10.5 ) {$\ddots$};

\end{tikzpicture}
\end{center}

\begin{rem}\label{fedor decomposition remark} Passing to a further subsequence in Theorem \ref{fedor decomposition theorem},  we may choose $\epsilon>0$ and assume that
\begin{align}\label{bmcmdmyz12}\begin{split}
&\quad \left\||b_m|-y^{\frac12}\right\|_{\mathcal{I}},\left\||c_m^{\ast}|-z^{\frac12}\right\|_{\mathcal{I}},\left\|d_m\right\|_{\mathcal{I}}\\
&\leq (2C_{\mathcal{I}})^{-m}\max\left\{\left\|y^{\frac12}\right\|_{\mathcal{I}},\left\|z^{\frac12}\right\|_{\mathcal{I}},\epsilon\right\},\quad m\geq 1. 
\end{split}
\end{align}

\end{rem}

Denote
$$P_n=\sum_{k=0}^{n-1}E_{k,k},\quad n\in\mathbb{Z}_+.$$

\begin{proof}[Proof of Theorem \ref{fedor decomposition theorem}]
 Without loss of generality, 
 we may assume that $$\left\|x_j\right\|_{\mathcal{I}}\leq 1$$ for every $j\geq 1.$ In other words, $\left\|x_j^{\ast}x_j\right\|_{\mathcal{I}^2},~\left\|x_jx_j^{\ast}\right\|_{\mathcal{I}^2}\leq 1$ for every $j\geq1.$ In particular, we have $$\left\|x_j^{\ast}x_j\right\|_{\infty},~\left\|x_jx_j^{\ast}\right\|_{\infty}\leq 1$$ for every $j\geq1.$ 
 Recall that the unit ball in $B(H)$ is compact in the weak operator topology\cite[Theorem 5.1.3]{KR}. Hence, the sequences $\left\{x_j^{\ast}x_j\right\}_{j\geq1}$ and $\left\{x_jx_j^{\ast}\right\}_{j\geq1}$ are compact in  the weak operator topology. 
There exists a subnet $\{j_i\}_i$ of the positive integers such that   $x_{j_i}^{\ast}x_{j_i}\to_i y$ and $x_{j_i}x_{j_i}^{\ast}\to_i z$ in the 
 weak operator topology.
 Since $\left(\mathcal{I},\left\|\cdot\right\|_{\mathcal{I}}\right)$ possesses the Fatou property, it follows that so does $\left(\mathcal{I}^2,\left\|\cdot\right\|_{\mathcal{I}^2}\right)$\footnote{
Here, 
 $\mathcal{I}^2=\{x\in B(H): |x|^{1/2} \in \mathcal{I}\}, ~\norm{x}_{\mathcal{I}^2 } =\norm{|x|^{1/2}}_{\mathcal{I}}^2$, 
 $I$ stands for the sequence space corresponding to $\cI$ and
 $I^2=\{x\in B(H): |x|^{1/2} \in I\}, ~\norm{x}_{I^2 } =\norm{|x|^{1/2}}_{I}^2 $. 
 The Fatou property of $\mathcal{I}^2$ follows from the fact that the sequence space  $I^2$  has the Fatou property, see e.g.  the argument in \cite[Theorem 6.1.7(iii)]{DPS}.}. By the Fatou property, we have
 $\left\|y\right\|_{\mathcal{I}^2},\left\|z\right\|_{\mathcal{I}^2}\leq1.$ In other words, 
 $$y^{\frac12},z^{\frac12}\in\mathcal{I} \mbox{~and~} \left\|y^{\frac12}\right\|_{\mathcal{I}},~\left\|z^{\frac12}\right\|_{\mathcal{I}}\leq 1.$$

Set $n_0=0$ and fix $n_1>1$ such  that (the constant $c''_\cI$ is given in Lemma \ref{abs holder estimate})
$$\left\|y^{\frac12}(1-P_{n_1})\right\|_{\mathcal{I}}\leq (c_{\mathcal{I}}'')^{-2}\cdot 1^{-2},\quad \left\|z^{\frac12}(1-P_{n_1})\right\|_{\mathcal{I}}\leq (c_{\mathcal{I}}'')^{-2}\cdot 1^{-2}.$$
Lemma \ref{abs holder estimate} yields
$$\left\|y^{\frac12}-|y^{\frac12}P_{n_1}|\right\|_{\mathcal{I}},~\left\|z^{\frac12}-|z^{\frac12}P_{n_1}|\right\|_{\mathcal{I}}\leq 1^{-1}.$$
Note that
$$P_{n_1}x_{j_i}^{\ast}x_{j_i} P_{n_1}\to_i P_{n_1}yP_{n_1},\quad P_{n_1}x_{j_i}x_{j_i}^{\ast}P_{n_1}\to_i P_{n_1}zP_{n_1},\quad P_{n_1}x_{j_i}P_{n_1}\to_i 0$$ 
in the  weak operator topology. For  finite-dimensional spaces  $P_{n_1} \cI P_{n_1}$ and  $P_{n_1} \cI^2  P_{n_1}$, convergence in the weak operator topology 
yields the convergence in $\cI$ and $\mathcal{I}^2.$ 
Hence,  we have 
$$P_{n_1}x_{j_i}^{\ast}x_{j_i} P_{n_1}\stackrel{\mathcal{I}^2}{\to}_i P_{n_1}yP_{n_1},\quad P_{n_1}x_{j_i}x_{j_i}^{\ast}P_{n_1}\stackrel{\mathcal{I}^2}{\to}_i P_{n_1}zP_{n_1},\quad P_{n_1}x_{j_i} P_{n_1}\stackrel{\mathcal{I}}{\to}_i 0.$$ 
In other words,
$$|x_{j_i} P_{n_1}|^2\stackrel{\mathcal{I}^2}{\to}_i |y^{\frac12}P_{n_1}|^2,\quad |x_{j_i}^{\ast}P_{n_1}|^2\stackrel{\mathcal{I}^2}{\to}_i |z^{\frac12}P_{n_1}|^2,\quad P_{n_1}x_{j_i}P_{n_1}\stackrel{\mathcal{I}}{\to}_i0. $$  
By Lemma \ref{bks lemma}, we have\footnote{Indeed,   we have $\norm{\left|x_jP_{n1}\right| -\left|y^{\frac12}P_{n_1}\right|}_\cI \stackrel{L.\ref{bks lemma}}{\le} c'_{\cI,2} \norm{ \left| |x_jP_{n_1}|^2 - |y^{\frac12}P_{n_1}|^2\right|^{1/2}}_\cI = c'_{\cI,2} \norm{   |x_jP_{n_1}|^2 - |y^{\frac12}P_{n_1}  |^2 }_{\cI^2 }^{1/2}  $.}  
$$|x_{j_i}P_{n_1}|\stackrel{\mathcal{I}}{\to}_i|y^{\frac12}P_{n_1}|,\quad |x_{j_i}^{\ast}P_{n_1}|\stackrel{\mathcal{I}}{\to}_i|z^{\frac12}P_{n_1}|,\quad P_{n_1}x_{j_i}P_{n_1}\stackrel{\mathcal{I}}{\to}_i 0.$$ 
Fix $j_1$ such that
$$\Big\||x_{j_1}P_{n_1}|-|y^{\frac12}P_{n_1}|\Big\|_{\mathcal{I}},~\Big\||x_{j_1}^{\ast}P_{n_1}|-|z^{\frac12}P_{n_1}\Big\|_{\mathcal{I}},~\Big\|P_{n_1}x_{j_1}P_{n_1}\Big\|_{\mathcal{I}}\leq 1^{-1}.$$
Fix $n_2>n_1$ such that 
$$\left\|(1-P_{n_2})x_{j_1}\right\|_{\mathcal{I}}~,\left\|(1-P_{n_2})x_{j_1}^{\ast}\right\|_{\mathcal{I}}\leq 1^{-1},$$
$$\left\|y^{\frac12}(1-P_{n_2})\right\|_{\mathcal{I}},~\left\|z^{\frac12}(1-P_{n_2})\right\|_{\mathcal{I}}\leq (c_{\mathcal{I}}'')^{-2}\cdot 2^{-2}.$$
By
Lemma \ref{abs holder estimate}, we have
$$\left\|y^{\frac12}-|P_{n_2}y^{\frac12}|\right\|_{\mathcal{I}},~\left\|z^{\frac12}-|P_{n_2}z^{\frac12}|\right\|_{\mathcal{I}}\leq 2^{-1}.$$
Arguing as above, we have 
$$P_{n_2}x_{j_i}^{\ast}x_{j_i} P_{n_2}\to_i P_{n_2}yP_{n_2},\quad P_{n_2}x_{j_i}x_{j_i}^{\ast}P_{n_2}\to_i P_{n_2}zP_{n_2},\quad P_{n_2}x_{j_i}P_{n_2}\to_i 0$$ 
in the weak operator topology;
% For finite-dimensional spaces $P_{n_2}\cI P_{n_2} $ and $P_{n_2}\cI^2 P_{n_2} $, convergence in the  weak operator topology yields the convergence in $\cI$ and  $\mathcal{I}^2.$ Hence, 
% $$P_{n_2}x_{j_i}^{\ast}x_{j_i}P_{n_2}\stackrel{\mathcal{I}^2}{\to}_i P_{n_2}yP_{n_2},\quad P_{n_2}x_{j_i}x_{j_i}^{\ast}P_{n_2}\stackrel{\mathcal{I}^2}{\to}_i P_{n_2}zP_{n_2},\quad P_{n_2}x_{j_i}P_{n_2}\stackrel{\mathcal{I}}{\to}_i 0,$$ 
% i.e., 
% $$|x_{j_i}P_{n_2}|^2\stackrel{\mathcal{I}^2}{\to}_i |y^{\frac12}P_{n_2}|^2,\quad |x_{j_i}^{\ast}P_{n_2}|^2\stackrel{\mathcal{I}^2}{\to}_i |z^{\frac12}P_{n_2}|^2,\quad P_{n_2}x_{j_i} P_{n_2}\stackrel{\mathcal{I}}{\to}_i0;$$
% By Lemma \ref{bks lemma}, we have
$$|x_{j_i} P_{n_2}|\stackrel{\mathcal{I}}{\to}_i|y^{\frac12}P_{n_2}|,\quad |x_{j_i}^{\ast}P_{n_2}|\stackrel{\mathcal{I}}{\to}_i|z^{\frac12}P_{n_2}|,\quad P_{n_2}x_{j_i} P_{n_2}\stackrel{\mathcal{I}}{\to}_i 0.$$ 
  Fix $j_2>j_1$ such that
$$\Big\||x_{j_2}P_{n_2}|-|y^{\frac12}P_{n_2}|\Big\|_{\mathcal{I}},\Big\||x_{j_2}^{\ast}P_{n_2}|-|z^{\frac12}P_{n_2}|\Big\|_{\mathcal{I}},\Big\|P_{n_2}x_{j_2}P_{n_2}\Big\|_{\mathcal{I}}\leq 2^{-1}.$$
Fix $n_3>n_2$ such that 
$$\|(1-P_{n_3})x_{j_2}\|_{\mathcal{I}},~\|(1-P_{n_3})x_{j_2}^{\ast}\|_{\mathcal{I}}\leq 2^{-1},$$
$$\left\|y^{\frac12}(1-P_{n_3})\right\|_{\mathcal{I}},~\left\|z^{\frac12}(1-P_{n_3})\right\|_{\mathcal{I}}\leq (c_{\mathcal{I}}'')^{-2}\cdot 3^{-2}.$$

Arguing inductively, we obtain a strictly increasing $\mathbb{N}$-valued sequence $\left\{j_m\right\}_{m\geq 1}$ and a strictly increasing $\mathbb{Z}_+$-valued sequence $\left\{n_m \right\}_{m\geq 0}$ such that $n_0=0$ and
$$\left\|y^{\frac12}-|y^{\frac12}P_{n_m}|\right\|_{\mathcal{I}},~\left\|z^{\frac12}-|z^{\frac12}P_{n_m}|\right\|_{\mathcal{I}}\leq m^{-1},$$
$$\Big\||x_{j_m}P_{n_m}|-|y^{\frac12}P_{n_m}|\Big\|_{\mathcal{I}},~\Big\||x_{j_m}^{\ast}P_{n_1}|-|z^{\frac12}P_{n_m}|\Big\|_{\mathcal{I}},~\Big\|P_{n_m}x_{j_m}P_{n_m}\Big\|_{\mathcal{I}}\leq m^{-1},$$
$$\left\|(1-P_{n_{m+1}})x_{j_m}\right\|_{\mathcal{I}},~\left\|(1-P_{n_{m+1}})x_{j_m}^{\ast}\right\|_{\mathcal{I}}\leq m^{-1},$$
for every $m\geq1.$ This, in particular, yields,
\begin{equation}\label{fed eq0}
|x_{j_m}P_{n_m}|-y^{\frac12},~|x_{j_m}^{\ast}P_{n_m}|-z^{\frac12},~P_{n_m}x_{j_m}P_{n_m}\stackrel{\mathcal{I}}{\to}0,\quad m\to\infty,
\end{equation}
and
\begin{equation}\label{fed eq1}
(1-P_{n_{m+1}})x_{j_m},~x_{j_m}(1-P_{n_{m+1}})\stackrel{\mathcal{I}}{\to}0,\quad m\to\infty.
\end{equation}

Set $p_m=P_{n_{m+1}}-P_{n_m},$ $m\geq1,$ and 
$$a_m=p_mx_{j_m}p_m,~ b_m=p_mx_{j_m}P_{n_m},~ c_m=P_{n_m}x_{j_m}p_m,~ d_m = x_{j_m}-a_m-b_m-c_m,~ m\geq 1.$$
Properties \eqref{deca} and \eqref{decb} of the decomposition follow from the definition of $a_m$ and $b_m,c_m$. It remains to establish the convergence in \eqref{decc}.

We denote, for brevity
$$a_m'=(1-P_{n_m})x_{j_m}(1-P_{n_m}),\quad b_m'=x_{j_m}P_{n_m},\quad c_m'=P_{n_m}x_{j_m}.$$
Note that 
\begin{align*}
a_m&=a_m'+(1-P_{n_{m+1}})x_{j_m}(1-P_{n_{m+1}})\\
&\qquad -
(1-P_{n_{m+1}})x_{j_m}\cdot (1-P_{n_m})-(1-P_{n_m})\cdot x_{j_m}(1-P_{n_{m+1}}), 
\end{align*} 
$$b_m=b_m'-P_{n_m}x_{j_m}P_{n_m}-(1-P_{n_{m+1}})x_{j_m}\cdot P_{n_m}$$
and 
$$c_m=c_m'-P_{n_m}x_{j_m}P_{n_m}-P_{n_m}\cdot x_{j_m}(1-P_{n_{m+1}}).$$
By \eqref{fed eq0} and \eqref{fed eq1}, we have 
\begin{equation}\label{fed eq2}
a_m'-a_m,b_m'-b_m,c_m'-c_m\stackrel{\mathcal{I}}{\to}0,\quad m\to\infty. 
\end{equation}
Moreover,
$$a_m'+b_m'+c_m'=x_{j_m}+P_{n_m}x_{j_m}P_{n_m}.$$ 
By \eqref{fed eq0},
 we have 
 $a_m'+b_m'+c_m'-x_{j_m}\to0$ in $\mathcal{I}$ as $m\to\infty.$ By \eqref{fed eq2}, $a_m+b_m+c_m-x_{j_m}\to0$ in $\mathcal{I}$ as $m\to\infty.$ In other words, $d_m\to0$ in $\mathcal{I}$ as $m\to\infty.$

By \eqref{fed eq0}, we have $|b_m'|-y^{\frac12},|(c_m')^{\ast}|-z^{\frac12}\to0$ in $\mathcal{I}$ as $m\to\infty.$ By \eqref{fed eq2} and Lemma \ref{abs holder estimate}, $|b_m|-|b_m'|,|c_m^{\ast}|-|(c_m')^{\ast}|\to0$ in $\mathcal{I}$ as $m\to\infty.$ Thus, $|b_m|\to y^{\frac12}$ and $|c_m^{\ast}|\to z^{\frac12}$ in $\mathcal{I}$ as $m\to\infty.$ This completes the proof.
\end{proof} 

We end this section with the following two useful estimates (see their usages in the   proofs of Lemmas \ref{fs lemma}, \ref{general new lemma}  respectively). 

\begin{lem}\label{main corollary} 
Let $\left(\mathcal{I},\left\|\cdot\right\|_{\mathcal{I}}\right)$ be a separable quasi-Banach ideal in $B(H)$ possessing the Fatou property. 
Let $a_m,b_m,c_m, d_m,y,z$ be as in   Theorem \ref{fedor decomposition theorem} (taking the rate of convergence specified in Remark \ref{fedor decomposition remark} into account). 
For any ${\bm \alpha} =({\bm \alpha}(m))_{m\ge 1}\in \ell_2$, we 
  have
\begin{align*}
&~\quad \left\|\sum_{m\geq1}{\bm \alpha}(m)b_m\right\|_{\mathcal{I}},~\left\|\sum_{m\geq1}{\bm \alpha}(m)c_m\right\|_{\mathcal{I}},~\left\|\sum_{m\geq1}{\bm \alpha}(m)d_m\right\|_{\mathcal{I}}\\
&\leq 2 C_{\mathcal{I}} \left\|{\bm \alpha}\right\|_2\max\left\{\left\|y^{\frac12}\right\|_{\mathcal{I}},\left\|z^{\frac12}\right\|_{\mathcal{I}},\epsilon\right\} .
\end{align*}
\end{lem}
\begin{proof}
We only prove the first estimate. The other two estimates follow from similar arguments. 
 Since the sequence $\{b_m\}_{m\geq1}$ is disjointly supported from the left, it follows that
$$\left|\sum_{m\geq1}{\bm \alpha}(m)b_m\right|=\left(\sum_{m\geq1} \left|{\bm \alpha}(m)b_m\right|^2\right)^{\frac12}.$$
Hence, we have 
$$\left|\sum_{m\geq1}{\bm \alpha}(m)b_m\right|\otimes E_{1,1}=\left(\sum_{m\geq1}|{\bm \alpha}(m)b_m|^2\right)^{\frac12}\otimes E_{1,1}=\left|\sum_{m\geq1} \left|{\bm \alpha}(m)b_m\right|\otimes E_{m,1}\right|.$$
Thus, by the quasi-triangle inequality of $\cI$, we have
\begin{eqnarray*}
& &\left\|\sum_{m\geq1}{\bm \alpha}(m)b_m\right\|_{\mathcal{I}}\\
&=&\left\|\sum_{m\geq1}|{\bm \alpha}(m)b_m|\otimes E_{m,1}\right\|_{\mathcal{I}}\\
&\stackrel{\eqref{qt infinitely many summands}}{\leq}& C_{\mathcal{I}}\cdot\left\|\sum_{m\geq1}|{\bm \alpha}(m)|y^{\frac12}\otimes E_{m,1}\right\|_{\mathcal{I}}+\sum_{m\geq1}C_{\mathcal{I}}^{m+1}|{\bm \alpha}(m)|\cdot \left\|(|b_m|-y^{\frac12})\otimes E_{m,1}\right\|_{\mathcal{I}}\\
&\stackrel{\eqref{bmcmdmyz12}}{\leq}  & C_{\mathcal{I}}\cdot \left\|{\bm \alpha}\right\|_2\left\|y^{\frac12}\right\|_{\mathcal{I}}+C_{\mathcal{I}}\cdot\left\|{\bm \alpha}\right\|_{\infty}\cdot\sum_{m\geq1}2^{-m}\cdot\max\left\{\left\|y^{\frac12}\right\|_{\mathcal{I}},\left\|z^{\frac12}\right\|_{\mathcal{I}},\epsilon\right\}\\
&\leq & 2C_{\mathcal{I}}\left\|{\bm \alpha}\right\|_2\max\left\{\left\|y^{\frac12}\right\|_{\mathcal{I}},\left\|z^{\frac12}\right\|_{\mathcal{I}},\epsilon\right\}.
\end{eqnarray*}
This completes the proof.
\end{proof}

\begin{lem}\label{aux corollary} Let $\left(\mathcal{I},\left\|\cdot\right\|_{\mathcal{I}}\right)$ be a separable quasi-Banach ideal in $B(H)$ possessing the Fatou property. 
Let $a_m$ be as in   Theorem \ref{fedor decomposition theorem}. We have
$$\left\|\sum_{m\geq1}{\bm \alpha}(m)a_m\right\|_{\infty}\leq \left\|{\bm \alpha}\right\|_{\infty}\cdot\sup_{j\geq1}\left\|x_j\right\|_{\mathcal{I}},~\forall {\bm \alpha} =({\bm \alpha}(m))_{m\ge 1}\in \ell_2.$$
\end{lem}
\begin{proof}
It suffices to note that
\begin{align*}
\left\|\sum_{m\geq1}{\bm \alpha}(m)a_m\right\|_{\infty}&=\left\|\bigoplus_{m\geq1}{\bm \alpha}(m)a_m\right\|_{\infty}=\sup_{m\geq1}|{\bm \alpha}(m)|\left\|a_m\right\|_{\infty}\\
&\leq \sup_{m\geq1}|{\bm \alpha}(m)|\left\|x_{j_m}\right\|_{\infty}\leq \left\|{\bm \alpha}\right\|_{\infty}\cdot\sup_{j\geq1}\left\|x_j\right\|_{\mathcal{I}}.
\end{align*}
\end{proof}
	
\section{Proof of Theorem \ref{main thm} for $p_1>2$}\label{first case section}
Theorem~\ref{main thm} for the case  when $p_1<2$ is treated in Section \ref{sec:holub}. 
In this section, we prove   Theorem~\ref{main thm} for the case when $p_1>2.$ 

%{\color{red} to Fedor: according to Jinghao, terminology ( $q$-Carothers-Dilwroth property) in the definition below is bad. this property should not be attributed to Carothers and Dilworth, but to somebody else. please, suggest a better terminology.}

\begin{defi} We say that a quasi-Banach ideal $(\mathcal{I},\left\|\cdot\right\|_{\mathcal{I}})$ possesses  property ($q$) if there exist constants $c_{\mathcal{I}}^{\uparrow}<\infty$ and $c_{\mathcal{I}}^{\downarrow}>0$ such that, for every sequence $\{A_M\}_{M\geq1}\subset\mathcal{I}$ converging to $0$ in the uniform norm, there exists a strictly increasing $\mathbb{N}$-valued sequence $\{M_n\}_{n\geq1}$ such that 
$$c_{\mathcal{I}}^{\downarrow}\left\|{\bm \alpha}\right\|_{\ell_q}\cdot\inf_{M\geq1}\left\|A_M \right\|_{\mathcal{I}}\leq\left\|\sum_{n\geq1}{\bm \alpha}(n) A_{M_n}\right\|_{\mathcal{I}}\leq c_{\mathcal{I}}^{\uparrow}\left\|{\bm \alpha}\right\|_{\ell_q}\cdot\sup_{M\geq1}\left\|A_M\right\|_{\mathcal{I}}.$$
\end{defi}
For example, $C_{p,q}$-ideals  possess the above property for $0<p,q<\infty $, see Lemma~\ref{CD lemma}.

For convenience, we introduce the following measure for $\mathbb{N}^2$:
for every finite set $A\subset\mathbb{N}^2,$ we 
$$S(A)=\sum_{M\geq1}\frac1M \left|\left\{A\cap(\{M\}\times\mathbb{N})\right\}\right|,$$
which will be useful in the proof of Lemmas \ref{fs lemma} and \ref{new verification lemma}  below. 
The following lemma is the key   auxiliary result  of this section.
\begin{lem}\label{fs lemma} Let $\left(\mathcal{I},\left\|\cdot\right\|_{\mathcal{I}}\right)$ be a separable quasi-Banach ideal having the Fatou property. Suppose   that $\left(\mathcal{I},\left\|\cdot\right\|_{\mathcal{I}}\right)$ possesses  the property ($q$). Let $\Theta,\Delta,c $  be positive scalars, $p> 2$ and let $\left\{x_{j}^{(M)}\right\}_{M,j\geq1}\subset\mathcal{I}$ be a bounded sequence such that
\begin{enumerate}[{\rm (i)}]
\item\label{sfla} for every finite set $A\subset\mathbb{N}^2,$
$$c\cdot \max \left\{S(A)^{\frac1p}\Delta,S(A)^{\frac12}\Theta\right\}\leq \left\|\sum_{(M,j)\in A}x_{j}^{(M)}\right\|_{\mathcal{I}}\leq \max\left\{S(A)^{\frac1p}\Delta,S(A)^{\frac12}\Theta\right\};$$
\item\label{sflb} for every $M\geq1,$ we have $x_{j}^{(M)}\to0$ in the weak operator topology as $j\to\infty;$
\end{enumerate}  
If $q\neq p,$ then there exists a constant $C_{c,p,\mathcal{I}}$ depending on $c,p, \cI$ only
such that 
$$\Delta\leq C_{c,p,\mathcal{I}}\cdot \Theta.$$
\end{lem}
\begin{proof} For every $M\geq1,$ it is assumed in \eqref{sflb} that $x_{j}^{(M)}\to0$ in the weak operator topology as $j\to\infty.$ For every $M\geq 1,$ applying Theorem \ref{fedor decomposition theorem} to the sequence $\{x^{(M)}_{j}\}_{j\geq1},$ we choose a strictly increasing $\mathbb{N}$-valued sequence $\{j_{m}^{(M)}\}_{m\geq 1}$ and a decomposition
\begin{align}\label{xMdecomposition}
x^{(M)}_{j_{m}^{(M)}}=a_{m}^{(M)}+b^{(M)}_{m}+c_{m}^{(M)}+d_{m}^{(M)},\quad m\geq1.
\end{align}
Let $y_M$ and $z_M$ be given by Theorem \ref{fedor decomposition theorem}. For every $M\geq 1,$ 
set  $\epsilon=C_{\mathcal{I}}^2M^{-\frac12}\Theta$ in  \eqref{bmcmdmyz12}  of Remark \ref{fedor decomposition remark}.
    By \eqref{sfla} (in particular, we have $\norm{\sum_{j\in A\subset \{M\}\times \mathbb{N} }x^{(M)}_j }_\cI \le  \left( \frac{|A|}{M}\right)^{1/2}\Theta $ for sufficiently large $|A|$) and Theorem \ref{limit point theorem}, we have  
\begin{equation}\label{sfl eq0}
\left\|y_M^{\frac12}\right\|_{\mathcal{I}},~\left\|z_M^{\frac12}\right\|_{\mathcal{I}}\leq C_{\mathcal{I}}^2M^{-\frac12}\Theta,\quad M\geq 1.
\end{equation}
We now write
$$\sum_{m=1}^Mx^{(M)}_{j^{(M)}_{m}}=A_M+B_M+C_M+D_M,$$
where (see \eqref{xMdecomposition})
$$A_M=\sum_{m=1}^Ma_{m}^{(M)},\quad B_M=\sum_{m=1}^Mb^{(M)}_{m},\quad C_M=\sum_{m=1}^Mc_{m}^{(M)},\quad D_M=\sum_{m=1}^Md_{m}^{(M)}.$$
By Lemma \ref{main corollary} and \eqref{sfl eq0}, we have
\begin{equation}\label{sfl eq1}
\left\|B_M\right\|_{\mathcal{I}},~\left\|C_M\right\|_{\mathcal{I}},~\left\|D_M\right\|_{\mathcal{I}}\leq 2C_{\mathcal{I}}M^{\frac12}\max\left\{\left\|y_M^{\frac12}\right\|_{\mathcal{I}},~\left\|z_M^{\frac12}\right\|_{\mathcal{I}},~\epsilon\right\}\leq 2C_{\mathcal{I}}^3\Theta.
\end{equation}

By the quasi-triangle inequality of $\cI$, we have 
\begin{eqnarray*}
\left\|A_M\right\|_{\mathcal{I}}&=&\left\|\sum_{m=1}^Mx^{(M)}_{ j^{(M)}_{m}}-B_M-C_M-D_M\right\|_{\mathcal{I}}\\
&\leq &
C_{\mathcal{I}}^2\left(\left\|\sum_{m=1}^Mx^{(M)}_{j^{(M)}_{m}}\right\|_{\mathcal{I}}+\left\|B_M\right\|_{\mathcal{I}}+\left\|C_M\right\|_{\mathcal{I}}+\left\|D_M\right\|_{\mathcal{I}}\right)\\
&
\stackrel{\eqref{sfla},\eqref{sfl eq1}}{\leq}& C_{\mathcal{I}}^2\Big(\max\{\Delta,\Theta\}+6C_{\mathcal{I}}^3\Theta\Big)\\
&\leq & 7C_{\mathcal{I}}^5\max\{\Delta,\Theta\}
\end{eqnarray*}
and 
\begin{eqnarray*}
c\cdot \max\{\Delta,\Theta\}&
\stackrel{\eqref{sfla}}{\leq}& \left\|\sum_{m=1}^Mx^{(M)}_{j^{(M)}_{m}}\right\|_{\mathcal{I}}\\
&=&\left\|A_M+B_M+C_M+D_M\right\|_{\mathcal{I}}\\
&\leq & 
C_{\mathcal{I}}^2\left(\left\|A_M\right\|_{\mathcal{I}}+\left\|B_M\right\|_{\mathcal{I}}+\left\|C_M\right\|_{\mathcal{I}}+\left\|D_M\right\|_{\mathcal{I}}\right)\\&\stackrel{\eqref{sfl eq1}}{\leq}&
 C_{\mathcal{I}}^2\left(\left\|A_M\right\|_{\mathcal{I}}+6C_{\mathcal{I}}^3\Theta\right).
\end{eqnarray*} 
Hence, we have 
$$\left\|A_M\right\|_{\mathcal{I}}\geq c\cdot C_{\mathcal{I}}^{-2}\cdot \max\left\{\Delta,\Theta\right\}-6C_{\mathcal{I}}^3\Theta\geq c\cdot C_{\mathcal{I}}^{-2}\Delta-6C_{\mathcal{I}}^3\Theta.$$
By Lemma \ref{aux corollary}, we have
$$\left\|A_M\right\|_{\infty}\le \sup_{j\ge 1}\norm{x^{(M)}_{j^{(M)}_{m}}}_\cI \stackrel{\eqref{sfla}}{\leq} \max\left\{M^{-\frac1p}\Delta,M^{-\frac12}\Theta\right\}.$$
By property ($q$), there exists a strictly increasing $\mathbb{N}$-valued sequence $\{M_n\}_{n\geq1}$ such that
\begin{equation}\label{sfl eq2}
c_{\mathcal{I}}^{\downarrow}\left\|{\bm \alpha}\right\|_{\ell_q}\big(c\cdot C_{\mathcal{I}}^{-2}\Delta-6C_{\mathcal{I}}^3\Theta\big)\leq \left\|\sum_{n\geq1}{\bm \alpha}(n)A_{M_n}\right\|_{\mathcal{I}}\leq 7c_{\mathcal{I}}^{\uparrow}C_{\mathcal{I}}^5\left\|{\bm \alpha}\right\|_{\ell_q}\max\{\Delta,\Theta\}
\end{equation}
for any ${\bm \alpha}=({\bm \alpha}(n))_{n\ge 1}\in \ell_q$. 

Now, we prove the assertion for the case when $q>p.$ For every $N\in\mathbb{N},$ we have
\begin{eqnarray*}
cN^{\frac1p}\Delta&\leq& c\cdot\max\{N^{\frac1p}\Delta,N^{\frac12}\Theta\}\stackrel{\eqref{sfla}}{\leq} \left\|\sum_{n=1}^N\sum_{m=1}^{M_n}x^{(M_n)}_{j^{(M_n)}_{m}}\right\|_{\mathcal{I}}\\
&=&\left\|\sum_{n=1}^NA_{M_n}+B_{M_n}+C_{M_n}+D_{M_n}\right\|_{\mathcal{I}}\\
&\leq& C_{\mathcal{I}}^2\cdot\left(\left\|\sum_{n=1}^NA_{M_n}\right\|_{\mathcal{I}}+\left\|\sum_{n=1}^N B_{M_n}\right\|_{\mathcal{I}}+\left\|\sum_{n=1}^NC_{M_n}\right\|_{\mathcal{I}}+\left\|\sum_{n=1}^ND_{M_n}\right\|_{\mathcal{I}}\right)\\
&\stackrel{\eqref{qt many summands}}{\leq}& C_{\mathcal{I}}^2\cdot\left(\left\|\sum_{n=1}^NA_{M_n}\right\|_{\mathcal{I}}+(2N)^{\log_2(C_{\mathcal{I}})}\sum_{n=1}^N\left(\left\| B_{M_n}\right\|_{\mathcal{I}}+\left\|C_{M_n}\right\|_{\mathcal{I}}+\left\|D_{M_n}\right\|_{\mathcal{I}}\right)\right)\\
&\stackrel{\eqref{sfl eq1},\eqref{sfl eq2}}{\leq}& 7c_{\mathcal{I}}^{\uparrow}C_{\mathcal{I}}^5N^{\frac1q}\max\{\Delta,\Theta\}+6C_{\mathcal{I}}^6N^{\log_2(C_{\mathcal{I}})+1}\Theta.
\end{eqnarray*}
Hence, we have   
%Omitting the intermediate terms and replacing maximum with the sum, we write
$$c\cdot N^{\frac1p}\Delta\leq 7c_{\mathcal{I}}^{\uparrow}C_{\mathcal{I}}^5N^{\frac1q}(\Delta+\Theta)+6C_{\mathcal{I}}^6N^{\log_2(C_{\mathcal{I}})+1}\Theta.$$
Set
$$N_1=\min\{N\in\mathbb{N}:\ cN^{\frac1p}>7c_{\mathcal{I}}^{\uparrow}C_{\mathcal{I}}^5N^{\frac1q}\},\quad C_{c,p,\mathcal{I}}=\frac{7c_{\mathcal{I}}^{\uparrow}C_{\mathcal{I}}^5N_1^{\frac1q}+6C_{\mathcal{I}}^6N_1^{\log_2(C_{\mathcal{I}})+1}}{c\cdot N_1^{\frac1p}-7c_{\mathcal{I}}^{\uparrow}C_{\mathcal{I}}^5N_1^{\frac1q}},$$
we obtain $\Delta\leq C_{c,p,\mathcal{I}}\Theta.$ This completes the proof for the case when  $q>p.$

Now, we prove the assertion for the case when $q<p.$ For every $N\in\mathbb{N},$ we have
\begin{eqnarray*}
& & c_{\mathcal{I}}^{\downarrow}N^{\frac1q}\left(c\cdot C_{\mathcal{I}}^{-2}\Delta-6C_{\mathcal{I}}^3\Theta\right)\\
&\stackrel{\eqref{sfl eq2}}{\leq}&\left\|\sum_{n=1}^NA_{M_n}\right\|_{\mathcal{I}}\\
& =&\left\|\sum_{n=1}^N\sum_{m=1}^{M_n}x^{(M_n)}_{j^{(M_n)}_{m}}-\sum_{n=1}^NB_{M_n}-\sum_{n=1}^NC_{M_n}-\sum_{n=1}^ND_{M_n}\right\|_{\mathcal{I}}\\
&\leq& C_{\mathcal{I}}^2\left(\left\|\sum_{n=1}^N\sum_{m=1}^{M_n}x^{(M_n)}_{j^{(M_n)}_{m}}\right\|_{\mathcal{I}}+\left\|\sum_{n=1}^NB_{M_n}\right\|_{\mathcal{I}}+\left\|\sum_{n=1}^NC_{M_n}\right\|_{\mathcal{I}}+\left\|\sum_{n=1}^ND_{M_n}\right\|_{\mathcal{I}}\right)\\
&\stackrel{\eqref{qt many summands}}{\leq}& C_{\mathcal{I}}^2\left(\left\|\sum_{n=1}^N\sum_{m=1}^{M_n}x^{(M_n)}_{ j^{(M_n)}_{m}}\right\|_{\mathcal{I}}+(2N)^{\log_2(C_{\mathcal{I}})}\sum_{n=1}^N\left(\left\|B_{M_n}\right\|_{\mathcal{I}}+\left\|C_{M_n}\right\|_{\mathcal{I}}+\left\|D_{M_n}\right\|_{\mathcal{I}} \right)\right)\\
&\stackrel{\eqref{sfl eq1}}{\leq}& C_{\mathcal{I}}^2\left\|\sum_{n=1}^N\sum_{m=1}^{M_n}x^{(M_n)}_{j^{(M_n)}_{m}}\right\|_{\mathcal{I}}+6C_{\mathcal{I}}^6N^{\log_2(C_{\mathcal{I}})+1}\Theta\\
&\stackrel{\eqref{sfla}}{\leq}& C_{\mathcal{I}}^2\max\{N^{\frac1p}\Delta,N^{\frac12}\Theta\}+6C_{\mathcal{I}}^6N^{\log_2(C_{\mathcal{I}})+1}\Theta.
\end{eqnarray*}
Therefore, we have   
%Omitting the intermediate terms and replacing maximum with the sum, we write
$$c_{\mathcal{I}}^{\downarrow}N^{\frac1q}\big(c\cdot C_{\mathcal{I}}^{-2}\Delta-6C_{\mathcal{I}}^3\Theta\big)\leq C_{\mathcal{I}}^2(N^{\frac1p}\Delta+N^{\frac12}\Theta)+6C_{\mathcal{I}}^6N^{\log_2(C_{\mathcal{I}})+1}\Theta.$$
Setting 
$$N_2=\min\left\{N\in\mathbb{N}:\ c_{\mathcal{I}}^{\downarrow}c N^{\frac1q}>C_{\mathcal{I}}^4N^{\frac1q}\right\}$$
and 
$$C_{c,p,\mathcal{I}}=\frac{6c_{\mathcal{I}}^{\downarrow}c C_{\mathcal{I}}^5N_2^{\frac1q}+C_{\mathcal{I}}^4N_2^{\frac12}+6C_{\mathcal{I}}^8N_2^{\log_2(C_{\mathcal{I}})+1}}{c_{\mathcal{I}}^{\downarrow}c N_2^{\frac1q}-C_{\mathcal{I}}^4N_2^{\frac1q}},$$
we obtain $\Delta\leq C_{c,p,\mathcal{I}}\Theta.$ This completes the proof for the case when $q<p.$
\end{proof}
Below, we construct a sequence in $\cI= \mathcal{C}_{p_2,q_2}$ satisfying the assumptions in  Lemma \ref{fs lemma} provided that the function space $ (L_{p_1,q_1}\cap L_2) (0,\infty)$ embeds isomorphically into $  \mathcal{C}_{p_2,q_2}$.
Recall that the quasi-norm of an element $x\in  (L_{p_1,q_1}\cap L_2) (0,\infty)$ is given by (see e.g. \cite[p.404]{Mastylo} and \cite[Chapter 7]{DPS})
\begin{align}\label{intersectionnorm}
\norm{x}_{ (L_{p_1,q_1}\cap L_2) (0,\infty)}:= \max\{  \norm{x}_{L_{p_1,q_1}},\norm{x }_{L_2}\}. 
\end{align}
\begin{lem}\label{new verification lemma} Let $2<p_1<\infty$ and $0<p_2,q_1,q_2<\infty.$ Suppose
$$T:(L_{p_1,q_1}\cap L_2)((0,\infty)^3)\to \mathcal{C}_{p_2,q_2}$$
is an isomorphic embedding. For every $f\in L_{\infty}(0,1)$ (in particular, $f\otimes\chi_{(j,j+1)}\otimes\chi_{(M,M+\frac1M)}\in (L_{p_1,q_1}\cap L_2)((0,\infty)^3)$), the sequence 
$$\left\{x_{j}^{(M)}\right\}_{M,j\geq1}=\left
\{T\left(f\otimes\chi_{(j,j+1)}\otimes\chi_{(M,M+\frac1M)}\right)
\right\}_{M,j\geq1}$$
satisfies the assumptions \eqref{sfla} and \eqref{sflb} in Lemma \ref{fs lemma} with $p=p_1,$ $q=q_2,$ $\Theta=\left\|f\right\|_2,$ $\Delta=\left\|f\right\|_{p_1,q_1}$ and 
$$c=\left\|T^{-1}\right\|_{T((L_{p_1,q_1}\cap L_2)((0,\infty)^3))\to (L_{p_1,q_1}\cap L_2)((0,\infty)^3)}^{-1}.$$
Here, we assume
$$\|T\|_{(L_{p_1,q_1}\cap L_2)((0,\infty)^3)\to  \mathcal{C}_{p_2,q_2}}=1.$$
\end{lem}
\begin{proof}
For any finite  set $A\subset \mathbb{N}^2$, 
by the definition of $S(A)$,  the Lebesgue measure of the support (in $(0,\infty)^2$) of $\sum_{M,j\in A} \chi_{(j,j+1)}\otimes\chi_{(M,M+\frac1M)}$ is equal to  $S(A)$, and therefore, 
we have
\begin{eqnarray*}
\left\|\sum_{M,j\in A}x_{j}^{(M)} \right\|_{ \mathcal{C}_{p_2,q_2}}&=&\left\|T\left(\sum_{M,j\in A}f\otimes\chi_{(j,j+1)}\otimes\chi_{(M,M+\frac1M)}\right)\right\|_{ \mathcal{C}_{p_2,q_2}}\\
&\leq&\left\|\sum_{M,j\in A}f\otimes\chi_{(j,j+1)}\otimes\chi_{(M,M+\frac1M)}\right\|_{(L_{p_1,q_1}\cap L_2)((0,\infty)^3)}\\
&=&\left\|f\otimes\chi_{(0,S(A))}\right\|_{(L_{p_1,q_1}\cap L_2)((0,\infty)^2)}\\
&\stackrel{\eqref{intersectionnorm}}{=}&\max\left\{S(A)^{\frac1{p_1}}\left\|f\right\|_{p_1,q_1},S(A)^{\frac12}\left\|f\right\|_2\right\}\\
&=&\max\left\{S(A)^{\frac1p}\Delta,S(A)^{\frac12}\Theta\right\}
\end{eqnarray*} 
and
\begin{eqnarray*}
\left\|\sum_{M,j\in A}x_{j}^{(M)} \right\|_{ \mathcal{C}_{p_2,q_2}}&=&\left\|T\left(\sum_{M,j\in A}f\otimes\chi_{(j,j+1)}\otimes\chi_{(M,M+\frac1M)}\right)\right\|_{ \mathcal{C}_{p_2,q_2}}\\
&\geq& c\cdot\left\|\sum_{M,j\in A}f\otimes\chi_{(j,j+1)}\otimes\chi_{(M,M+\frac1M)}\right\|_{(L_{p_1,q_1}\cap L_2)((0,\infty)^3)}\\
&=&c\cdot \left\|f\otimes\chi_{(0,S(A))}\right\|_{(L_{p_1,q_1}\cap L_2)((0,\infty)^2)}\\
&=&c\cdot \max\left\{S(A)^{\frac1{p_1}}\left\|f\right\|_{p_1,q_1},S(A)^{\frac12}\left\|f\right\|_2\right\}\\
&=&c\cdot \max\left\{
S(A)^{\frac1p}\Delta,S(A)^{\frac12}\Theta\right\}.
\end{eqnarray*}
This proves the assumption \eqref{sfla} in Lemma \ref{fs lemma}.

Fix $r>p_1>2.$ For every bounded linear functional $F$ on $(L_r\cap L_2)((0,\infty)^3),$ there exists $\psi\in (L_{\frac{r}{r-1}}+L_2)((0,\infty)^3)$\cite{Loza} (see also \cite{SemSuk}) such that
$$F:g\to \int \psi g,\quad g\in (L_r\cap L_2)((0,\infty)^3).$$
Hence, the sequence $\{f\otimes\chi_{(j,j+1)}\otimes\chi_{(M,M+\frac1M)}\}_{j\geq1}$ converges weakly to $0$ in the space $(L_r\cap L_2)((0,\infty)^3)$ as $j\to\infty.$

Since $r>p_1>2,$ it follows that $(L_r\cap L_2)((0,\infty)^3)\subset (L_{p_1,q_1}\cap L_2)((0,\infty)^3).$ Fix $\xi,\eta\in \ell_2$ and define a bounded linear functional $F_{\xi,\eta}$ on $(L_r\cap L_2)((0,\infty)^3)$ by the formula
$$F_{\xi,\eta}:g\to \langle T(g)\xi,\eta\rangle,\quad g\in (L_r\cap L_2)((0,\infty)^3).$$
It follows from the preceding paragraph that
$$\langle x^{(M)}_{j}\xi,\eta\rangle=F_{\xi,\eta}(f\otimes\chi_{(j,j+1)}\otimes\chi_{(M,M+\frac1M)})\to0,\quad j\to\infty.$$
Since $\xi,\eta\in \ell_2$ are arbitrary, it follows that $x_{j}^{(M)} \to0$ in the weak operator topology as $j\to\infty.$ This proves the assumption \eqref{sflb} in Lemma \ref{fs lemma}.
\end{proof}

 Now, we show that the isomorphic  embedding $(L_{p_1,q_1}\cap L_2)(0,\infty)$ into $ \mathcal{C}_{p_2,q_2}$  implies that  $p_1=q_2.$
 
\begin{lem}\label{final lemma} Let $p_1,q_1,p_2,q_2>0.$ If $p_1>2$ and $(L_{p_1,q_1}\cap L_2)(0,\infty)\hookrightarrow \mathcal{C}_{p_2,q_2},$ then $p_1=q_2.$
\end{lem}
\begin{proof} Fix an isomorphic embedding
$$T:(L_{p_1,q_1}\cap L_2)((0,\infty)^3)\to \mathcal{C}_{p_2,q_2}.$$
Without loss of generality, we may assume that 
$$\left\|T\right\|_{(L_{p_1,q_1}\cap L_2)((0,\infty)^3)\to  \mathcal{C}_{p_2,q_2}}=1.$$
Denote for brevity $V:=T((L_{p_1,q_1}\cap L_2)((0,\infty)^3))\subset  \cC_{p_2,q_2}$,  
$$c_T=\left\|T^{-1}\right\|_{V\to (L_{p_1,q_1}\cap L_2)((0,\infty)^3)}^{-1}$$
and $c_{p_2,q_2}=C_{\mathcal{C}_{p_2,q_2}}$ (the constant for the quasi-triangular inequality of $\cC_{p_2,q_2}$).

Let $f\in L_{\infty}(0,1).$ By Lemma \ref{new verification lemma}, the sequence
$$\left \{x_{ j}^{(M)} \right \}_{M,j\geq1}=\Big\{T\left(f\otimes\chi_{(j,j+1)}\otimes\chi_{(M,M+\frac1M)}\right)\Big\}_{M,j\geq1}$$
satisfies the assumptions \eqref{sfla} and \eqref{sflb} in Lemma \ref{fs lemma} with $p=p_1,$ $q=q_2,$ $\Theta=\left\|f\right\|_2,$ $\Delta=\left\|f\right\|_{p_1,q_1}$ and $c=c_T.$ If $p_1\neq q_2,$ then it follows from Lemma \ref{fs lemma} (with $\cI=\cC_{p_2,q_2}$) that
$$\left\|f\right\|_{p_1,q_1}=\Delta\leq C_{c_T,p_1,\mathcal{C}_{p_2,q_2}}\Theta=C_{c_T,p_1,\mathcal{C}_{p_2,q_2}}\left\|f\right\|_2.$$
As $f\in L_{\infty}(0,1)$ is arbitrary, the latter inequality is impossible. This shows that $p_1\neq q_2$ is impossible.
\end{proof}

\begin{proof}[Proof of Theorem \ref{main thm} for $p_1>2$] Suppose $L_{p_1,q_1}(0,1)\hookrightarrow\mathcal{C}_{p_2,q_2}.$
	
By Lemma \ref{estimate for Kf} and Lemma \ref{boundedness of K}, the Kruglov operator $K$ is bounded from $L_{p_1,q_1}(0,1)$ into $ L_{p_1,q_1}(0,1).$ By Theorem \ref{AS thm}, the space $Z_{L_{p_1,q_1}}^2(0,\infty)=(L_{p_1,q_1}\cap L_2)(0,\infty)$ embeds isomorphically into $L_{p_1,q_1}(0,1).$ Hence, there exists an isomorphic embedding of the space $(L_{p_1,q_1}\cap L_2)(0,\infty)$ into $\mathcal{C}_{p_2,q_2}.$ Using Lemma \ref{final lemma}, we conclude that $p_1=q_2.$

By Theorem \ref{easy thm}, we either have $p_1=q_1=2$ or $q_1=q_2.$ By the preceding paragraph and the assumption that $p_1>2$, we have   $p_1=q_1=q_2>2.$
\end{proof}

\section{Proof of Theorem \ref{main thm} for $p_1=2$}
In \cite{AHS}, it is proved that $L_{2,q}(0,1)\not\hookrightarrow C_{2,q}$ when $1\le q\ne 2 <\infty$.
In this section, using different approaches, 
we extend this result to a much more general setting that $L_{2,q_1}(0,1)\not\rightarrow C_{p_2,q_2}$ when $0< q_1\ne 2<\infty $. 
 
\begin{lem}\label{general new lemma} If $Z_{L_{2,q}}^2(0,\infty)\hookrightarrow\mathcal{C}_{p,q},$
then, for every $f\in L_{2,q}(0,1),$ there exists an element $a\in \ell_{p,q}$ such that either
$$\left\|a\otimes{\bm \alpha}\right\|_{\ell_{p,q}}+\left\|{\bm \alpha}\right\|_{\ell_q}\lesssim \left\|f\otimes{\bm \alpha}\right\|_{Z_{L_{2,q}}^2((0,\infty)\times\mathbb{Z}_+)}+\left\|{\bm \alpha}\right\|_{\ell_2}\lesssim \left\|{\bm \alpha}\right\|_{\ell_2}+\left\|a\otimes{\bm \alpha}\right\|_{\ell_{p,q}}+\left\|{\bm \alpha}\right\|_{\ell_q}$$
or
$$\left\|a\otimes{\bm \alpha}\right\|_{\ell_{p,q}}\lesssim \left\|f\otimes{\bm \alpha}\right\|_{Z_{L_{2,q}}^2((0,\infty)\times\mathbb{Z}_+)}+\left\|{\bm \alpha}\right\|_{\ell_2}\lesssim \left\|{\bm \alpha}\right\|_{\ell_2}+\left\|a\otimes{\bm \alpha}\right\|_{\ell_{p,q}}$$
for all finite sequence ${\bm \alpha}$\footnote{For ${\bm \alpha}=(\alpha_1,\alpha_2,\cdots, \alpha_n,0,\cdots)$, we identify $a\otimes {\rm \alpha}$ with the direct sum $\oplus_{i=1}^n \alpha_i a $.}.
\end{lem}
\begin{proof} Fix an isomorphic embedding $T:Z_{L_{2,q}}^2((0,\infty)\times\mathbb{Z}_+)\to\mathcal{C}_{p,q}.$ 
Without loss of generality, we may assume that 
$$\left\|T\right\|_{Z_{L_{2,q}}^2((0,\infty)\times\mathbb{Z}_+)\to  \mathcal{C}_{p,q}}=1.$$
Denote for brevity $V:=T(Z_{L_{2,q}}^2((0,\infty)\times\mathbb{Z}_+)) \subset C_{p,q}$,  
$$c_T=\left\|T^{-1}\right\|_{V \to Z_{L_{2,q}}^2((0,\infty)\times\mathbb{Z}_+)}^{-1}$$
and $c_{p,q}=C_{\mathcal{C}_{p,q}}$ (the constant for the quasi-triangular inequality of $\cC_{p,q}$).
	
Take $f\in L_{2,q}(0,1)$ and set $x_j=T(f\otimes\chi_{(j,j+1)}),$ $j\geq1.$ By Theorem \ref{fedor decomposition theorem}, we have
$$x_{j_m}=a_m+b_m+c_m+d_m.$$
To lighten the notations, we assume that $\left\|y^{\frac12}\right\|_{\cC_{p,q}},\left\|z^{\frac12}\right\|_{\cC_{p,q}}\leq 1$ (here, $y$ and $z$ are given by Theorem \ref{fedor decomposition theorem}).

By the quasi-triangular inequality, we have 
\begin{eqnarray*}
& &c_{p,q}^{-2}\left\|
\sum_{m\geq1}{\bm \alpha}(m)\mu(a_m)\otimes e_m\right\|_{p,q}\\&=&c_{p,q}^{-2}\left\|\sum_{m\geq1}{\bm \alpha}(m)a_m\right\|_{\ell_{p,q}}\\
&=&c_{p,q}^{-2}\left\|\sum_{m\geq1}{\bm \alpha}(m)x_{j_m}-\sum_{m\geq1}{\bm \alpha}(m)b_m-\sum_{m\geq1}{\bm \alpha}(m)c_m-\sum_{m\geq1}{\bm \alpha}(m)d_m\right\|_{\cC_{p,q}}\\
&\leq&\left\|\sum_{m\geq1}{\bm \alpha}(m)x_{j_m}\right\|_{\cC_{p,q}}+\left\|\sum_{m\geq1}{\bm \alpha}(m)b_m\right\|_{\cC_{p,q}}+\left\|\sum_{m\geq1}{\bm \alpha}(m)c_m\right\|_{\cC_{p,q}}+\left\|\sum_{m\geq1}{\bm \alpha}(m)d_m\right\|_{\cC_{p,q}}
\end{eqnarray*}
and
\begin{eqnarray*}
& &c_{p,q}^{-2}\left\|\sum_{m\geq1}{\bm \alpha}(m)x_{j_m}\right\|_{\cC_{p,q}}\\
&=&c_{p,q}^{-2}\left\|\sum_{m\geq1}{\bm \alpha}(m)a_m+\sum_{m\geq1}{\bm \alpha}(m)b_m+\sum_{m\geq1}{\bm \alpha}(m)c_m+\sum_{m\geq1}{\bm \alpha}(m)d_m\right\|_{\cC_{p,q}}\\
&\leq& \left\|\sum_{m\geq1}{\bm \alpha}(m)a_m\right\|_{\cC_{p,q}}+\left\|\sum_{m\geq1}{\bm \alpha}(m)b_m\right\|_{\cC_{p,q}}+\left\|\sum_{m\geq1}{\bm \alpha}(m)c_m\right\|_{\cC_{p,q}}+\left\|\sum_{m\geq1}{\bm \alpha}(m)d_m\right\|_{\cC_{p,q}}.
\end{eqnarray*}
 Moreover, by the definition of $x_j$, we have
$$c_T\left\|f\otimes{\bm \alpha}\right\|_{Z_{L_{2,q}}^2((0,\infty)\times\mathbb{Z}_+)}\leq \left\|\sum_{m\geq1}{\bm \alpha}(m)x_{j_m}\right\|_{\cC_{p,q}}\leq \left\|f\otimes{\bm \alpha}\right\|_{Z_{L_{2,q}}^2((0,\infty)\times\mathbb{Z}_+)}.$$
By Lemma \ref{main corollary} (taking $\epsilon \le 1$), 
$$\left\|\sum_{m\geq1}{\bm \alpha}(m)b_m\right\|_{\cC_{p,q}},\left\|\sum_{m\geq1}{\bm \alpha}(m)c_m\right\|_{\cC_{p,q}},\left\|\sum_{m\geq1}{\bm \alpha}(m)d_m\right\|_{\cC_{p,q}}\leq 2c_{p,q}\left\|{\bm \alpha}\right\|_{\ell_2}.$$
Thus, combining estimates above, we have
$$\left\|\sum_{m\geq1}{\bm \alpha}(m)\mu(a_m)\otimes e_m\right\|_{\ell_{p,q}}\leq c_{p,q}^2\left(\left\|f\otimes{\bm \alpha}\right\|_{Z_{L_{2,q}}^2\left((0,\infty)\times\mathbb{Z}_+ \right)}+6c_{p,q}\left\|{\bm \alpha}\right\|_{\ell_2}\right)$$
and
\begin{align}\label{f tensor}
c_T\left\|f\otimes{\bm \alpha}\right\|_{Z_{L_{2,q}}^2((0,\infty)\times\mathbb{Z}_+)}\leq c_{p,q}^2\left( \left\|\sum_{m\geq1}{\bm \alpha}(m)\mu(a_m)\otimes e_m\right\|_{p,q}+6c_{p,q}\left\|{\bm \alpha}\right\|_2\right).
\end{align}

Passing to a subsequence if necessary, 
there exists an element $a\in \ell_{p,q}$ such that   $\mu(a_m)\to a$ in the uniform norm as $m\to \infty$ \footnote{ 
Note that $\mu(t; a_m)\lesssim \sup\norm{x_{j_m}}_{\cC_{p,q}} t^{-1/p}$, $t> 0$ \cite[Chapter 4, Proposition 4.2]{BS}, for all $m\ge 1$.
Observing  that  $\ell_{p,q}\subset c_0$ and applying 
\cite[Proposition 5.2]{DPS2016} to $\{ \{\mu(k;a_m)  \}_{k=0}^\infty \}_{m\ge 1}\subset c_0$, we obtain the existence of $a$ in $c_0$. By the Fatou property of $\ell_{p,q}$, $a\in\ell_{p,q}$. }. By the Fatou property of $\cC_{p,q}$, we have 
\begin{align}\label{a tensor}\begin{split}
\left\|a\otimes{\bm \alpha}\right\|_{\ell_{p,q}}&=\left\|\sum_{m\geq1}{\bm \alpha}(m)a\otimes e_m\right\|_{\ell_{p,q}}\\
&\leq  c_{p,q}^2\left(\left\|f\otimes{\bm \alpha}\right\|_{Z_{L_{2,q}}^2((0,\infty)\times\mathbb{Z}_+)}+6c_{p,q}\left\|{\bm \alpha}\right\|_2\right),
\end{split}
\end{align}
 which proves the left hand side of the second inequality in the assertion. 

Denote $y_m=\mu(a_m)-a,$ $m\geq1.$ By the quasi-triangular  inequality, we have 
\begin{align}\label{ym tensor}
\left\|\sum_{m\geq1}{\bm \alpha}(m)y_m\otimes e_m\right\|_{\ell_{p,q}}\leq 2c_{p,q}^3\Big(\|f\otimes{\bm \alpha}\|_{Z_{L_{2,q}}^2((0,\infty)\times\mathbb{Z}_+)}+6c_{p,q}\left\|{\bm \alpha}\right\|_{\ell_2}\Big).
\end{align}
By construction, $y_m\to0$ in the uniform norm as $m\to \infty $. Passing to a subsequence if necessary, we may assume that there exists a non-negative number  $\delta$ such that $\left\|y_m\right\|_{\ell_{p,q}}\to \delta$ as $m\to\infty$.  Passing to a further subsequence and using Lemma \ref{original CD lemma}, we obtain
\begin{enumerate}
\item $$(2c_{p,q}^{(1)})^{-1}\left\|{\bm \alpha}\right\|_{\ell_q}\cdot\delta\leq\left\|\sum_{m\geq1}{\bm \alpha}(m)y_m\otimes e_m\right\|_{\ell_{p,q}}\leq 2c_{p,q}^{(1)}\left\|{\bm \alpha}\right\|_{\ell_q}\cdot\delta$$
if $\delta>0$;
\item and 
$$\left\|\sum_{m\geq1}{\bm \alpha}(m)y_m\otimes e_m\right\|_{\ell_{p,q}}\leq\left\|{\bm \alpha}\right\|_{\infty}$$
if $\delta=0.$
\end{enumerate}
  This together with \eqref{a tensor} and \eqref{ym tensor} yields
  that
  \begin{align*}
&\quad   \left\|a\otimes{\bm \alpha}\right\|_{\ell_{p,q}}+\left\|{\bm \alpha}\right\|_{\ell_q}\\
 & \le c_{p,q}^2\left(\left\|f\otimes{\bm \alpha}\right\|_{Z_{L_{2,q}}^2((0,\infty)\times\mathbb{Z}_+)}+6c_{p,q}\left\|{\bm \alpha}\right\|_2\right)  +    \frac{2c_{p,q}^{(1)}}{\delta } \left\|\sum_{m\geq1}{\bm \alpha}(m)y_m\otimes e_m\right\|_{\ell_{p,q}}  \\
  &\lesssim \left\|f\otimes{\bm \alpha}\right\|_{Z_{L_{2,q}}^2((0,\infty)\times\mathbb{Z}_+)}+\left\|{\bm \alpha}\right\|_{\ell_2} , 
  \end{align*}
i.e.,   the left hand side of the first inequality in the assertion is proved.

By the quasi-triangular inequality of $\ell_{p,q} $ (see e.g. \cite{BS}), we have
\begin{enumerate}
\item \begin{align*}
\left\|\sum_{m\geq1}{\bm \alpha}(m)\mu(a_m)\otimes e_m\right\|_{\ell_{p,q}}&\leq c_{p,q}\left(\left\|\sum_{m\geq1}{\bm \alpha}(m)y_m\otimes e_m\right\|_{\ell_{p,q}}+\left\|\sum_{m\geq1}{\bm \alpha}(m)a\otimes e_m\right\|_{\ell_{p,q}}\right)\\
&\leq c_{p,q}\left(2c_{p,q}^{(1)}\left\|{\bm \alpha}\right\|_{\ell_q}\cdot\delta+\left\|a\otimes{\bm \alpha}\right\|_{\ell_{p,q}}\right)
\end{align*}
if $\delta>0$;
\item and
$$\left\|\sum_{m\geq1}{\bm \alpha}(m)\mu(a_m)\otimes e_m\right\|_{\ell_{p,q}}\leq c_{p,q}\Big(\left\|{\bm \alpha}\right\|_{\infty}+\left\|a\otimes{\bm \alpha}\right\|_{\ell_{p,q}}\Big)$$
if $\delta=0.$
\end{enumerate}
  Hence, by \eqref{f tensor}, we have 
\begin{enumerate}
\item $$c_T\|f\otimes{\bm \alpha}\|_{Z_{L_{2,q}}^2((0,\infty)\times\mathbb{Z}_+)}\leq c_{p,q}^3\Big(2c_{p,q}^{(1)}\left\|{\bm \alpha}\right\|_{\ell_q}\cdot\delta+\|a\otimes{\bm \alpha}\|_{\ell_{p,q}}+6\left\|{\bm \alpha}\right\|_{\ell_2}\Big)$$
if $\delta>0$;
\item and
$$c_T\|f\otimes{\bm \alpha}\|_{Z_{L_{2,q}}^2((0,\infty)\times\mathbb{Z}_+)}\leq c_{p,q}^3\Big(\left\|{\bm \alpha}\right\|_{\infty}+\|a\otimes{\bm \alpha}\|_{\ell_{p,q}}+6\left\|{\bm \alpha}\right\|_{\ell_2}\Big)$$
if $\delta=0.$
\end{enumerate}  
In particular, these two inequalities yield    the right hand sides of the  inequalities in the assertion.
\end{proof}

\begin{lem}\label{p=2 compute lemma} Let $p>0$ and $q>2.$ If $f\in L_{2,q}(0,1)$ and $a\in \ell_{p,q}$ are such that
\begin{align}
\left\|a\otimes{\bm \alpha}\right\|_{\ell_{p,q}}
&\lesssim \left\|f\otimes{\bm \alpha}\right\|_{L_{2,q}}+\left\|f\otimes{\bm \alpha}\right\|_{L_1+L_2}+\left\|{\bm \alpha}\right\|_{\ell_2}\label{72eq1}\\
&\lesssim\left\|{\bm \alpha}\right\|_{\ell_2}+\left\|a\otimes{\bm \alpha}\right\|_{\ell_{p,q}}+\left\|{\bm \alpha}\right\|_{\ell_q}\label{72eq2}
\end{align}
for all finite sequence ${\bm \alpha}$,
then $f\in L_2(0,1).$
\end{lem}
\begin{proof} Let $p\geq 2.$
For all finite scalar sequences ${\bm \alpha}$,   we have
$$\left\|f\otimes{\bm \alpha}\right\|_{L_1+L_2}\stackrel{\eqref{72eq2}}{\lesssim} \left\|{\bm \alpha}\right\|_{\ell_2}+\left\|a\otimes{\bm \alpha}\right\|_{\ell_{p,q}} +\left\|{\bm \alpha}\right\|_{\ell_q}.$$
Setting ${\bm \alpha}=(\underbrace{1,\cdots,1}_{n-\rm times},0,\cdots)$ and dividing by $n^{\frac12},$ we write
$$n^{-\frac12}\left\|\sigma_nf\right\|_{L_1+L_2}\lesssim 1+n^{\frac1p-\frac12}\left\|a\right\|_{\ell_{p,q}}+n^{\frac1q-\frac12}.$$
In particular,
$$\sup_{n\geq1}n^{-\frac12}\left\|\sigma_nf\right\|_{L_1+L_2}<\infty.$$
Hence,
noting that 
$n^{-\frac12} \norm{\sigma_nf}_{L_1+L_2}
\gtrsim  n^{-\frac{1}{2}} \left( \int_1^\infty \mu(s;\sigma_nf )^2 ds  \right)^{\frac{1}{2}}
=  \Big(\int_{\frac1n}^1\mu(s;f)^2 ds\Big)^{\frac12} \to \norm{f}_{L_2}
$
as $n\to \infty $, we have 
 $f\in L_2(0,1).$ This proves the assertion for the case when $p\geq 2.$

Let $p<2.$ We have
$$\left\|a\right\|_{\infty}\left\|{\bm \alpha}\right\|_{\ell_{p,q}}\leq \left\|a\otimes{\bm \alpha}\right\|_{\ell_{p,q}}\stackrel{\eqref{72eq1}}{\lesssim} \left\|f\otimes{\bm \alpha}\right\|_{L_{2,q}}+\left\|f\otimes{\bm \alpha}\right\|_{L_1+L_2}+\left\|{\bm \alpha}\right\|_{\ell_2}.$$
Setting ${\bm \alpha}=(\underbrace{1,\cdots,1}_{n-\rm times},0,\cdots)$,  we write
$$n^{\frac1p}\left\|a\right\|_{\infty}\lesssim n^{\frac12}\left\|f\right\|_{L_{2,q}} + \left\|\sigma_nf\right\|_{L_1+L_2}+n^{\frac12}.$$
However, 
\begin{align*}
\left\|\sigma_nf\right\|_{L_1+L_2}& \approx \int_0^1\mu\left(\frac{s}{n};f\right)ds+\left(\int_1^{\infty}\mu\left(\frac{s}{n};f\right)^2 ds\right)^{\frac12}\\
&=n\int_0^{\frac1n}\mu(s;f)ds+n^{\frac12}\Big(\int_{\frac1n}^1\mu(s;f)^2 ds\Big)^{\frac12}\\
&\leq \left\|f\right\|_{L_{2,\infty}}\cdot\left( n\int_0^{\frac1n}s^{-\frac12}ds+n^{\frac12}\Big(\int_{\frac1n}^1\frac{ds}{s}\Big)^{\frac12}\right)=O\left(
n^{\frac12}\log^{\frac12}(n)
\right).
\end{align*}
It follows that
$$n^{\frac1p}\left\|a\right\|_{\infty}\lesssim O\left(
n^{\frac12}\log^{\frac12}(n)\right).$$
Hence, $a=0.$ 
By assumption, we have
$$\left\|f\otimes{\bm \alpha}\right\|_{L_1+L_2}\lesssim\left\|{\bm \alpha}\right\|_{\ell_2}+\left\|{\bm \alpha}\right\|_{\ell_q}.$$
Setting ${\bm \alpha}=(\underbrace{1,\cdots,1}_{n-\rm times},0,\cdots)$  and dividing by $n^{\frac12},$ we write
$$n^{-\frac12}\left\|\sigma_nf\right\|_{L_1+L_2}\lesssim 1+n^{\frac1q-\frac12}.$$
In particular,
$$\sup_{n\geq1}n^{-\frac12}\left\|\sigma_nf\right\|_{L_1+L_2}<\infty.$$
Hence, $f\in L_2(0,1).$ This completes the proof for the case when  $p<2.$
\end{proof}
The following lemma should be compared with \cite[Example 3.5]{JSZ20}. 
\begin{lem}\label{idea behind ahs} Let $q\in(0,2).$
There exists no positive number $c_q$ depending on $q$ only such that 
the inequality
$$\left\|f\otimes{\bm \alpha}\right\|_{L_{2,q}+L_{\infty}}\leq c_q\left\|f\right\|_{L_{2,q}}\left\|{\bm \alpha}\right\|_{\ell_2},\quad f\in L_{2,q}(0,1),\quad {\bm \alpha}\in \ell_2,$$
holds. 
\end{lem}
\begin{proof} Fix $h=\mu(h)\in L_2(0,1)\backslash L_{2,q}(0,1)$ and set $f=\chi_{(0,\frac1n)}$ and 
set ${\bm \alpha}$ by
${\bm \alpha}(k)=n\int_{\frac{k}{n}}^{\frac{k+1}{n}}h$, $k=0,\cdots, n-1$.
We have  
$$\mu(f\otimes{\bm \alpha})\chi_{(0,1)}=\mu(f\otimes{\bm \alpha})=\sum_{k=0}^{n-1}\left(n\int_{\frac{k}{n}}^{\frac{k+1}{n}}h\right)\chi_{(\frac{k}{n},\frac{k+1}{n})}.$$
Thus,
\begin{align*}
\left\|
\sum_{k=0}^{n-1}\Big(n\int_{\frac{k}{n}}^{\frac{k+1}{n}}h\Big)\chi_{(\frac{k}{n},\frac{k+1}{n})}\right\|_{L_{2,q}}&=\left\|
\mu(f\otimes{\bm \alpha})\chi_{(0,1)}\right\|_{L_{2,q}}=
\left\|f\otimes{\bm \alpha}\right\|_{(L_{2,q}+L_{\infty})((0,\infty)\times\mathbb{Z}_+)}\\
&\leq c_q n^{\frac{1}{2}}\left\|{\bm \alpha}\right\|_{\ell_2}=c_q\left\|\sum_{k=0}^{n-1}\left(n\int_{\frac{k}{n}}^{\frac{k+1}{n}}h\right)\chi_{(\frac{k}{n},\frac{k+1}{n})}\right\|_{L_2}.
\end{align*}
Passing $n\to\infty,$ we conclude that $\left\|h\right\|_{L_{2,q}}\leq c_q\left\|h\right\|_{L_2}.$ However, the left hand side is assumed to be infinite.
\end{proof}
The following result can be found in \cite[Lemma 2]{Carothers81}, \cite[(8)]{Carothers}, \cite[(14)]{CT}  and  \cite[Theorem 7.4]{ON}. 
\begin{prop}\label{ONeil}Let $p\ge  q$ and 
let $x,y\in L_{p,q}(0,\infty )$.
We have 
$$\norm{x\otimes y}_{L_{p,q}}\le \norm{x }_{L_{p,q}}\norm{ y}_{L_{p,q}}. $$
\end{prop}
%\begin{proof}Let us first consider the case for $L_{r,1}$
%when $r> 1$.
%Let $x,y\in L_{r,1}(0,\infty )$.

%Fix $x\in L_{r,1}(0,\infty )$ and define the operator $T_x:y\to x\otimes y$. we claim that $\left\|T_x\right\|_{L_{r,1}\to L_{r,1}}\leq\left\|x\right\|_{L_{r,1}}$. 
%It suffices to verify the assertion for the extreme points of the unit ball of $L_{r,1}(0,\infty )$, that is, 
%for functions in the form of $m(A)^{-1/p}\chi_A$, where $A\subset [0,\infty)$ is a Lebesgue measurable set
% (see \cite[Theorem 1]{CT}). 
%Hence, it suffices to observe that  $\left\|x\otimes\chi_A\right\|_{L_{r,1}}\leq\left\|x\right\|_{L_{r,1}}\left\|\chi_A\right\|_{L_{r,1}}.$
%Therefore, 
%for any $y\in L_{r,1}(0,\infty)$,
%we have 
%$$\norm{x\otimes y}_{L_{r,1}}\le \norm{x }_{L_{r,1}}\norm{ y}_{L_{r,1}}. $$

%For $x,y\in L_{p,q}(0,\infty )$,
% we have 
%\begin{align*}\left\|x\otimes y \right\|_{L_{p,q}}
%&=\left( \int_0^\infty \mu(t; x\otimes y )^q d t^{q/p} \right)^{1/q}
%= \left( \int_0^\infty \mu(t; |x|^q\otimes |y|^q ) d t^{q/p} \right)^{1/q}  \\
%& =\left\||x|^q\otimes|y|^q\right\|_{L_{\frac{p}{q},1}}^{\frac1q}\leq \left\||x|^q\right\|_{L_{\frac{p}{q},1}}^{\frac1q}\left\||y|^q\right\|_{L_{\frac{p}{q},1}}^{\frac1q}=\left\|x\right\|_{L_{p,q}}\left\|y \right\|_{L_{p,q}},
%\end{align*}
%which completes the proof of the proposition. 
%\end{proof}

The following lemma should be compared with \cite[Lemma 7]{SS}.
\begin{lem}\label{one more idea behind ahs} Let $q\in(0,2).$ For every $f\in L_{2,q}(0,1),$ we have
$$\left\|f\otimes\{k^{-\frac12}\}_{k=1}^n\right\|_{L_{2,q}+L_{\infty}}=o(\log^{\frac1q}(n)),\quad n\to\infty.$$
\end{lem}
\begin{proof}
% For every $r\geq1,$ we have
%$$\left\|g\otimes{\bm \beta}\right\|_{L_{r,1}}
%\leq \left\|g\right\|_{L_{r,1}} \left\|{\bm \beta}\right\|_{\ell_{r,1} },\quad g\in L_{r,1}(0,1),\quad {\bm \beta}\in \ell_{r,1}.$$
%Hence, 
By Proposition \ref{ONeil} (with identifying elements from $\ell_{p,q}$ as functions in $L_{p,q}(0,\infty)$), for any $g\in L_{2,q}(0,1) $ and $ {\bm \alpha}\in \ell_{2,q}$, 
we have  
$$\left\|f\otimes{\bm \alpha}\right\|_{L_{2,q}}
%=\left\||f|^q\otimes|{\bm \alpha}|^q\right\|_{L_{\frac{2}{q},1}}^{\frac1q}\leq \left\||f|^q\right\|_{L_{\frac{2}{q},1}}^{\frac1q}\left\||{\bm \alpha}|^q\right\|_{L_{\frac{2}{q},1}}^{\frac1q}=
\le 
\left\|f\right\|_{L_{2,q}}\left\|{\bm \alpha}\right\|_{\ell_{2,q}}.$$
In particular, there exists a constant $c_q$ depending on $q$ only such that 
\begin{align*}
\left\|f\otimes\{k^{-\frac12}\}_{k=1}^n\right\|_{L_{2,q} }&\leq \left\|f\right\|_{L_{2,q}}\left\|\{k^{-\frac12}\}_{k=1}^n\right\|_{\ell_{2,q}}\\
&
=\left\|f\right\|_{L_{2,q}}\left( \int_0^\infty \mu(t; \{k^{-\frac12}\}_{k=1}^n )^q  dt^{q/2}\right)^{1/q }\\
&\le \left\|f\right\|_{L_{2,q}} \left(  1+ \frac{q}{2}
\sum_{k=2}^n  k^{-\frac{q}{2} }    (k-1)^{\frac{q}{2}-1 }  \right)^{1/q }
\\
&\le  \left  \|f\right\|_{L_{2,q}} \left(   
2 \sum_{k=1}^n  k^{-1 }   \right)^{1/q }
\\
&\leq c_q\left\|f\right\|_{L_{2,q}}\log^{\frac1q}(n).
\end{align*}

Fix $\epsilon>0$ and write $f=g+h,$ where $\left\|h\right\|_{L_{2,q} }<\epsilon$ and $g\in L_{\infty}(0,1).$ We have
\begin{align*}
\left\|f\otimes\{k^{-\frac12}\}_{k=1}^n\right\|_{L_{2,q}+L_{\infty}}&\leq
 \left\|h\otimes\{k^{-\frac12}\}_{k=1}^n\right\|_{L_{2,q}}+ \left\|g\otimes\{k^{-\frac12}\}_{k=1}^n\right\|_{\infty}\\
 &
 \leq c_q\epsilon\cdot\log^{\frac1q}(n)+\left\|g\right\|_{\infty}.
 \end{align*}
Hence,
$$\limsup_{n\to\infty}\log^{-\frac1q}(n)\left\|f\otimes\{k^{-\frac12}\}_{k=1}^n\right\|_{L_{2,q}+L_{\infty}}\leq c_q\epsilon.$$
Since $\epsilon>0$ is arbitrary, the assertion follows.
\end{proof}

\begin{lem}\label{p=2 second compute lemma} Let $p>0$ and $q\in(0,2).$ There exists $0\neq f\in L_{2,q}(0,1)$ such that 
there exists no  $a\in \ell_{p,q}$ satisfying 
$$\left\|a\otimes{\bm \alpha}\right\|_{\ell_{p,q}}+\left\|{\bm \alpha}\right\|_{\ell_q}\lesssim \|f\otimes{\bm \alpha}\|_{L_{2,q}+L_{\infty}}+\left\|{\bm \alpha}\right\|_{\ell_2}
\lesssim\left\|a\otimes{\bm \alpha}\right\|_{\ell_{p,q}}+\left\|{\bm \alpha}\right\|_{\ell_q}$$
or
$$\left\|a\otimes{\bm \alpha}\right\|_{\ell_{p,q}}\lesssim \left\|f\otimes{\bm \alpha}\right\|_{L_{2,q}+L_{\infty}}+\left\|{\bm \alpha}\right\|_{\ell_2}\lesssim\left\|a\otimes{\bm \alpha}\right\|_{\ell_{p,q}}+\left\|{\bm \alpha}\right\|_{\ell_2}$$
for all finite sequence ${\bm \alpha}$.
\end{lem}
\begin{proof} 
Assume by contradiction that 
for every $f\in L_{2,q}(0,1)$,  there exists 
 an element $a\in \ell_{p,q}$ such that
$$\left\|a\otimes{\bm \alpha}\right\|_{\ell_{p,q}}+\left\|{\bm \alpha}\right\|_{\ell_q}\lesssim \|f\otimes{\bm \alpha}\|_{L_{2,q}+L_{\infty}}+\left\|{\bm \alpha}\right\|_{\ell_2}
\lesssim\left\|a\otimes{\bm \alpha}\right\|_{\ell_{p,q}}+\left\|{\bm \alpha}\right\|_{\ell_q}$$
or
$$\left\|a\otimes{\bm \alpha}\right\|_{\ell_{p,q}}\lesssim \left\|f\otimes{\bm \alpha}\right\|_{L_{2,q}+L_{\infty}}+\left\|{\bm \alpha}\right\|_{\ell_2}\lesssim\left\|a\otimes{\bm \alpha}\right\|_{\ell_{p,q}}+\left\|{\bm \alpha}\right\|_{\ell_2}$$
for all finite sequence ${\bm \alpha}$.
	
In the first case, we have
$$\left\|{\bm \alpha}\right\|_{\ell_q}\lesssim \left\|f\otimes{\bm \alpha}\right\|_{L_{2,q}+L_{\infty}}+\left\|{\bm \alpha}\right\|_{\ell_2}.$$
Setting ${\bm \alpha}=\chi_{[0,n)},$ we write
$$n^{\frac1q}\lesssim \left\|\sigma_nf\right\|_{L_{2,q}+L_{\infty}}+n^{\frac12}\lesssim n^{\frac12}.$$
Hence, the first case is impossible and we only consider the second case.

If $p>2,$ then
$$\left\|f\otimes{\bm \alpha}\right\|_{L_{2,q}+L_{\infty}}\lesssim
\left\|a\otimes{\bm \alpha}\right\|_{\ell_{p,q}}+\left\|{\bm \alpha}\right\|_{\ell_2}\leq\left\|a\right\|_{\ell_{p,q}}\left\|{\bm \alpha}\right\|_{\ell_{p,q}}+\left\|{\bm \alpha}\right\|_{\ell_2}
\lesssim\left\|{\bm \alpha}\right\|_{\ell_2}.$$
Since such an inequality is true for {\it every} $f\in L_{2,q}(0,1),$ it follows from the Uniform Boundedness Principle that
$$\left\|f\otimes{\bm \alpha}\right\|_{L_{2,q}+L_{\infty}}\leq c_q\left\|f\right\|_{L_{2,q}}\left\|{\bm \alpha}\right\|_{\ell_2},\quad f\in L_{2,q}(0,1),\quad {\bm \alpha}\in \ell_2,$$
which contradicts Lemma \ref{idea behind ahs}.

If $p<2,$ then
$$\left\|a\right\|_{\infty}\left\|{\bm \alpha}\right\|_{\ell_{p,q}}\leq \left\|a\otimes{\bm \alpha}\right\|_{\ell_{p,q}}\lesssim \left\|f\otimes{\bm \alpha}\right\|_{L_{2,q}+L_{\infty}}+\left\|{\bm \alpha}\right\|_{\ell_2}.$$
Setting ${\bm \alpha}=\chi_{[0,n)},$ we have
$$n^{\frac1p}\left\|a\right\|_{\infty}\lesssim \left\|\sigma_nf\right\|_{L_{2,q}+L_{\infty}}+n^{\frac12}\leq \left\|\sigma_nf\right\|_{L_{2,q}}+n^{\frac12}=n^{\frac12}\left\|f\right\|_{L_{2,q}}+n^{\frac12}.$$
This means $a=0.$ Hence, we have
$$\left\|f\otimes{\bm \alpha}\right\|_{L_{2,q}+L_{\infty}}\lesssim\left\|{\bm \alpha}\right\|_{\ell_2}.$$
Since such an inequality is true for {\it every} $f\in L_{2,q}(0,1),$ it follows from the Uniform Boundedness Principle that
$$\left\|f\otimes{\bm \alpha}\right\|_{L_{2,q}+L_{\infty}}\leq c_q\left\|f\right\|_{L_{2,q}}\left\|{\bm \alpha}\right\|_{\ell_2},\quad f\in L_{2,q}(0,1),\quad {\bm \alpha}\in \ell_2.$$
This is impossible by Lemma \ref{idea behind ahs}.

If $p=2,$ then
$$\left\|a\otimes{\bm \alpha}\right\|_{\ell_{2,q}}\lesssim \left\|f\otimes{\bm \alpha}\right\|_{L_{2,q}+L_{\infty}}+\left\|{\bm \alpha}\right\|_{\ell_2}
\lesssim
\left\|a\otimes{\bm \alpha}\right\|_{L_{2,q}}+\left\|{\bm \alpha}\right\|_{\ell_2}.$$
If $a\neq 0,$ then we write
$$\left\|{\bm \alpha}\right\|_{\ell_{2,q}}\lesssim\left\|a\right\|_{\infty}\left\|{\bm \alpha}\right\|_{\ell_{2,q}}\leq\left\|a\otimes{\bm \alpha}\right\|_{\ell_{2,q}}\lesssim \left\|f\otimes{\bm \alpha}\right\|_{L_{2,q}+L_{\infty}}+\left\|{\bm \alpha}\right\|_{\ell_2}.$$
In particular, letting ${\bm \alpha}=\{k^{-\frac12}\}_{k=1}^n$, since $\norm{{\bm \alpha}}_{\ell_{2,q}}\gtrsim \norm{{\bm \alpha}}_{\ell_2}$, it follows that 
$$\left\|f\otimes\{k^{-\frac12}\}_{k=1}^n\right\|_{L_{2,q}+L_{\infty}}\gtrsim \log^{\frac1q}(n),\quad n\to\infty.$$
This contradicts Lemma \ref{one more idea behind ahs}. Thus, $a=0$ for every $f\in L_{2,q}(0,1)$ and we have
$$\left\|f\otimes{\bm \alpha}\right\|_{L_{2,q}+L_{\infty}}\lesssim\left\|{\bm \alpha}\right\|_{\ell_2}.$$
Since such an inequality is true for {\it every} $f\in L_{2,q}(0,1),$ it follows from the Uniform Boundedness Principle that
$$\left\|f\otimes{\bm \alpha}\right\|_{L_{2,q}+L_{\infty}}\leq c_q\left\|f\right\|_{2,q}\left\|{\bm \alpha}\right\|_2,\quad f\in L_{2,q}(0,1),\quad {\bm \alpha}\in \ell_2.$$
This contradicts Lemma \ref{idea behind ahs}.
\end{proof}

\begin{lem}\label{p=2 final lemma} If $p,q>0,$ $q\neq 2,$ then $L_{2,q}(0,1)\not\hookrightarrow \mathcal{C}_{p,q}.$
\end{lem}
\begin{proof} By Theorem \ref{AS thm}, $Z_{L_{2,q}}^2(0,\infty)\hookrightarrow L_{2,q}(0,1).$ 
Therefore, it   suffices to prove that  $Z_{L_{2,q}}^2(0,\infty)\not\hookrightarrow \mathcal{C}_{p,q}.$
	
If $q>2,$ then
$$Z_{L_{2,q}}^2(0,\infty)=L_{2,q}(0,\infty)\cap (L_1+L_2)(0,\infty),\quad q>2.$$
If $Z_{L_{2,q}}^2(0,\infty)\hookrightarrow \mathcal{C}_{p,q}$, then 
for  $0\ne f\in L_{2,q}(0,1)\setminus L_2(0,1),$ there exists an element $a\in \ell_{p,q}$ (see Lemma \ref{general new lemma} above) such that either
$$\left\|a\otimes{\bm \alpha}\right\|_{\ell_{p,q}}+\left\|{\bm \alpha}\right\|_{\ell_q}\lesssim \left\|f\otimes{\bm \alpha}\right\|_{Z_{L_{2,q}}^2((0,\infty)\times\mathbb{Z}_+)}+\left\|{\bm \alpha}\right\|_{\ell_2}\lesssim \left\|{\bm \alpha}\right\|_{\ell_2}+\left\|a\otimes{\bm \alpha}\right\|_{\ell_{p,q}}+\left\|{\bm \alpha}\right\|_{\ell_q}$$
or
$$\left\|a\otimes{\bm \alpha}\right\|_{\ell_{p,q}}\lesssim \left\|f\otimes{\bm \alpha}\right\|_{Z_{L_{2,q}}^2((0,\infty)\times\mathbb{Z}_+)}+\left\|{\bm \alpha}\right\|_{\ell_2}\lesssim \left\|{\bm \alpha}\right\|_{\ell_2}+\left\|a\otimes{\bm \alpha}\right\|_{\ell_{p,q}}.$$
Since $q>2$, the second inequality implies the first one. 
This contradicts   Lemma~\ref{p=2 compute lemma} above and therefore, 
the assertion  that $Z_{L_{2,q}}^2(0,\infty)\not\hookrightarrow \mathcal{C}_{p,q}$ follows.

If $q\in(0,2),$ then
$$Z_{L_{2,q}}^2(0,\infty)=(L_{2,q}+L_{\infty})(0,\infty)\cap L_2(0,\infty).$$
In this case, the assertion  that $Z_{L_{2,q}}^2(0,\infty)\not\hookrightarrow \mathcal{C}_{p,q}$ follows from a combination of Lemma~\ref{general new lemma} and Lemma \ref{p=2 second compute lemma}.
This completes the proof.
\end{proof}

\begin{proof}[Proof of Theorem \ref{main thm} for $p_1=2$] Suppose $p_1=2$ and $L_{p_1,q_1}(0,1)\hookrightarrow\mathcal{C}_{p_2,q_2}.$ By Theorem \ref{easy thm}, we have either $p_1=q_1=2$ or $q_1=q_2.$ In the first case, 
we have $L_{2,2}(0,1)=L_2(0,1)\approx \ell_2 \hookrightarrow \cC_{p_2,q_2}$;  in the second case, the assertion follows from Lemma~\ref{p=2 final lemma}.
\end{proof}

\begin{rem}
    In Theorem \ref{main thm},
if $p_1=q_1=2,$ then $L_{p_1,q_1}(0,1)=L_2(0,1)\hookrightarrow\mathcal{C}_{p_2,q_2}.$ 
%The case when  $p_1=q_1=q_2>2$  remains open.
%In other words, we do not know whether $L_q(0,1)\hookrightarrow\mathcal{C}_{p,q}$ for $q>2.$ Only the special case $2<p\leq q$ is answered (in the negative) in Proposition 4.1 in \cite{AHS}.
The case when %the Question \ref{quest2} remains open is
$p_1=q_1=q_2>2$ remains open.
In other words, we do not know whether $L_q(0,1)\hookrightarrow\mathcal{C}_{p,q}$ for $q>2.$ Only the case when $2<p\leq q$ is treated (in the negative) in Proposition 4.1 in \cite{AHS}.
\end{rem}

{\bf  Acknowledgement} The authors would like to thank Zhizheng Yu for helpful discussions.
      
     {\bf  Data Availability}  No data is used in the paper. 
     
     { \bf Conflict of interest}  All authors declare that we have no conflict of interest.


\begin{thebibliography}{99}


\bibitem{AK}
F. Albiac, N. Kalton,
{\it Topics in Banach space theory,}
Graduate Texts in Mathematics 233, Springer, 2006.

%\bibitem{ACL}
%Z. Altshuler, P. Casazza, B.-L.  Lin,
%{\it On symmetric Basic sequences in Lorentz sequence spaces},
%Israel. J. Math.  {15} (1973), 144--155.
 
% \bibitem{Arazy78}
%J. Arazy, {\it 
%Some remarks on interpolation theorems and the boundness of the triangular projection in unitary matrix spaces}, 
% Integral Equations Operator Theory 1 (1978), no. 4, 453--495.

\bibitem{Arazy81}
J. Arazy, {\it Basic sequences, embeddings, and the uniqueness of the symmetric structure in unitary matrix spaces,}
J. Funct. Anal. {  40} (1981), 302--340. 


\bibitem{ALin}
J. 
Arazy, P.  Lin, {\it  
On $p$-convexity and $q$-concavity of unitary matrix spaces},
 Int. Equ. Oper. Theory 8 (1985), 295--313. 

\bibitem{AL}
J. Arazy,  J. Lindenstrauss,
{\it Some linear topological properties of the spaces $C_p$ of operators on Hilbert space.}
Compositio Math. {  30} (1975), 81--111.

\bibitem{A11}
S. Astashkin, {\it 
Rademacher series and isomorphisms of rearrangement invariant spaces on the finite interval and on the semi-axis}, J. Funct. Anal. 260 (2011), 195--207.

\bibitem{AHS} S. Astashkin, J. Huang, F. Sukochev, {\it Lack of isomorphic embeddings of symmetric function
spaces into operator ideals}, J. Funct. Anal. 280 (2021), no. 5, 108895.


\bibitem{ASS}
S. Astashkin, E. Semenov, F. Sukochev,
{\it The Banach--Saks $p$-property},
Math. Ann.  {332} (2005), 879--900.


%\bibitem{AMS}
%S. Astashkin, L. Maligranda, E. Semenov,
%{\it Multiplicator space and complemented subspaces of rearrangement invariant space,}
%J. Funct. Anal.  {2002} (2003), 247--276.
 
\bibitem{AS04}
S. Astashkin, F. Sukochev, {\it 
Comparison of sums of independent and disjoint functions in symmetric spaces,} Mat. Zametki, 76(4) (2004),  483--489.

\bibitem{AS05}
 S. Astashkin, F. Sukochev, {\it 
 Series of independent random variables in rearrangement invariant spaces: an operator approach,}
  Israel J. Math. 145 (2005), 125--156.
  
\bibitem{AS07}
 S. Astashkin, F. Sukochev, {\it 
 Series of independent, mean zero random variables in rearrangement-invariant spaces having the Kruglov property,}
  (Russian) Zap. Nauchn. Sem. S.-Peterburg. Otdel. Mat. Inst. Steklov. (POMI) 345 (2007), Issled. po Line\v{ı}n. Oper. i Teor. Funkts. 34, 25--50, 140; translation in J. Math. Sci. (N.Y.) 148 (2008), no. 6, 795--809.  

\bibitem{AS-uspehi} S. Astashkin, F. Sukochev, {\it Independent functions and the geometry of Banach spaces,} Russian Math. Surveys  {65} (2010), no. 6, 1003--1081.

\bibitem{ASZ}
S. Astashkin, F. Sukochev, D. Zanin,
{\it 
Disjointification inequalities in symmetric quasi-Banach spaces and their applications,}
Pacific J. Math. 270(2) (2014), 257--285.

%\bibitem{Banach}
%S. Banach, {\it Theorie des operations lineaires,}
%Monographie Matematyczne 1, Warsaw, 1932. 
  
%\bibitem{BL} Bergh J., L\"ofstr\"om J. {\it Interpolation spaces. An introduction.} Grundlehren der Mathematischen Wissenschaften, No. 223. Springer-Verlag, Berlin-New York, 1976.
 
\bibitem{BS}
 C. Bennett, R. Sharpley, {\it Interpolation of Operators}, Academic Press, Boston, 1988.
  
% \bibitem{Bhatia}
% R. Bhatia, {\it
% Perturbation inequalities for the absolute value map in norm ideals of operators,}
%  J. Operator Theory 19 (1988) 129--136.
  
\bibitem{BCLS}
A. Ber, V. Chilin, G. Levitina, F. Sukochev,
{\it Derivations on symmetric quasi-Banach ideals of compact operators},
J. Math. Anal. Appl. {\bf 272}(12) (2017), 2984--2997.
  
\bibitem{CSZ} L. Cadilhac, F. Sukochev, D. Zanin,  {\it Lorentz--Shimogaki and Arazy--Cwikel theorems revisited,} Pacific J. Math.  {\bf 328} (2024), no. 2, 227--254.
 
 \bibitem{Carothers81}
N. Carothers,
{\it Rearrangement invariant subspaces of Lorentz function spaces,}
Israel J. Math.  {40} (1981), 217--228.

%\bibitem{Carothers2}
%N. Carothers,
%{\it Rearrangement invariant subspaces of Lorentz function spaces, II, }
%Rocky Mountain J. Math.  {17} (1987), 607--616.



%\bibitem{CL}
%P. Casazza, B.-L. Lin,
%{\it Projections on Banach spaces with symmetric bases,}
%Studia Math.  {52} (1974), 189--193.


\bibitem{Creekmore} J. Creekmore, {\it Type and cotype in Lorentz $L_{pq}$ spaces.}
Indag. Math.  {84} (2) (1981), 145--152.

\bibitem{Carothers}
N. Carothers, {
\it 
 Rearrangement invariant subspaces of Lorentz function spaces. II,
 } Rocky Mountain J. Math. 17 (1987), no. 3, 607--616.

\bibitem{CD85}
N. Carothers, S. Dilworth,
{\it Geometry of Lorentz spaces via interpolation},
Texas Functional Analysis Seminar 1985--1986 (Austin, TX, 1985--1986), 107--133, Longhorn Notes, Univ. Texas, Austin, TX, 1986.


\bibitem{CD88}
N. Carothers, S. Dilworth,
{\it Subspaces of $L_{p,q}$,}
Proc. Amer. Math. Soc.
{  104} (1988), no. 2, 537--545.
\bibitem{CD89}
N. Carothers, S. Dilworth,
{\it Equidistributed random variables in $L_{p,q}$,}
    J. Funct. Anal.
     {84} (1989), 146--159.

\bibitem{CT}
N. Carothers, B. Turett,  {\it
Isometries on $L_{p,1}$}, 
Trans. Amer. Math. Soc. 297 (1986), no. 1, 95--103.

%\bibitem{CF}
%N. Carothers, P. Flinn,
%{\it Embedding $l_p^{n^{\bm \alpha}}$ in $l_{p,q}^n$},
%Proc. Amer. Math. Soc.  {88} (1983), 523--526.




\bibitem{CDS}
V. Chilin, P. Dodds, F. Sukochev, {\it  The Kadec-Klee property in symmetric
spaces of measurable operators}, Israel J. Math. 97 (1997), 203--219,
 

%\bibitem{Creekmore}
%J. Creekmore, 
%{\it
%Type and cotype in Lorentz $L_{pq}$ spaces}, Indag. Math. 84 (2) (1981) 145--152.




\bibitem{Dilworth}
 S. Dilworth, {\it Special Banach lattices and their applications}, in: Handbook of the Geometry of BanachSpaces, vol. I, North-Holland, Amsterdam, 2001, pp. 497--532. 
 

%\bibitem{Dilworth90} 
%S. Dilworth, {\it A scale of linear spaces related to the $L_p$ scale,} 
%Illinois J. Math. {34} (1) (1990), 140--148.

%\bibitem{Dirksen} Dirksen S. {\it Noncommutative Boyd interpolation theorems,} Trans. Amer. Math. Soc. {\bf 367} (2015), no. 6, 4079--4110.

%\bibitem{DDPS}
%P.  Dodds, T. Dodds, B. de Pagter, F.  Sukochev, 
%{\it   Lipschitz continuity of the absolute value and Riesz projections in symmetric operator spaces}, 
 % J. Funct. Anal. 148 (1997), no. 1, 28--69.
  
% \bibitem{DDS07}
% P. Dodds, T. Dodds, F. Sukochev,
% {\it Banach--Saks properties in symmetric spaces of measurable operators,}
% Studia Math.  {178} (2) (2007), 125--166. 
% 
  \bibitem{DDS14}   P. Dodds, T. Dodds, F. Sukochev,
   {\it On $p$-convexity and $q$-concavity in non-commutative symmetric spaces},
Integr. Equ. Oper. Theory  {78} (2014), 91--114.


%\bibitem{DDP} P.G.~Dodds, T.K.-Y.~Dodds, B.~Pagter, {\it Non-Commutative Banach Function Spaces}, Math.\ Z. {  201} (1989), 583--597.

 



%\bibitem{DFPS}
% P. Dodds, S.  Ferleger, B. de Pagter,  F. Sukochev,
%  {\it Vilenkin systems and generalized triangular truncation operator,} Integr. Equ. Oper. Theory
%   {  40} (2001), no. 4, 403--435.
%\bibitem{DP2}
%P. Dodds, B. de Pagter, {\it Normed K\"{o}the spaces: A non-commutative viewpoint,}
%Indag. Math. 25 (2014), 206--249.

\bibitem{DPS2016} P. Dodds, B. de Pagter, F. Sukochev, {\it Sets of uniformly absolutely continuous norm in symmetric spaces of measurable operators,}
Trans. Amer. Math. Soc. 368 (6) (2016), 4315--4355.

 

\bibitem{DPS}
P. Dodds, B. de Pagter, F. Sukochev,
{\it Noncommutative Integration and Operator Theory, }
Progress in Mathematics, volume 349, Birkh\"auser, Cham, Switzerland, 2023. 



%\bibitem{DSS}
%P. Dodds, E. Semenov, F. Sukochev,
%{\it The Banach--Saks properties in rearrangement invariant spaces,}
% Studia
%Math.  {162} (2004), 263--294.

%\bibitem{Dunford}
%N. Dunford, {\it
%Spectral operators},
% Pacific J. Math. 4 (1954), 321--354.

%\bibitem{Fremlin2} Fremlin D. {\it Measure theory. Vol. 2.} Torres Fremlin, Colchester, 2003.

\bibitem{Friedman}
Y.
Friedman, {\it
Subspace of $LC(H)$ and $C_p$,}
Proc. Amer. Math. Soc. 53 (1975), no. 1, 117--122.

\bibitem{Garling}
D. J. H. Garling, {\it On ideals of operators in Hilbert space},
 Proc. London Math. Soc. 17 (1967), 115--138.

%\bibitem{GT83}
%D.J.H. Garling, N. 
% Tomczak-Jaegermann, {\it
%The cotype and uniform convexity of unitary ideals},
% Israel J. Math. 45 (1983), no. 2-3, 175--197.


 \bibitem{Gillespie}
  T. Gillespie, {\it
  Boundedness criteria for Boolean algebras of projections,}
   J. Funct. Anal.  {148}
(1997), 70--85.

\bibitem{GnedenkoKolmogorov} B. Gnedenko, A. Kolmogorov   {\it Limit distributions for sums of independent random variables.} Addison-Wesley Publishing Co., 1968.


\bibitem{GK}
I. Gohberg,  M.  Kre\v{i}n,  {\it 
Introduction to the theory of linear nonselfadjoint operators}, Translations of Mathematical Monographs, Vol. 18. American Mathematical Society, Providence, RI, 1969.

\bibitem{GM}
J. Grala-Michalak, A.  Michalak,
{\it
On constructions of isometric copies of $L_p(0,1)$ spaces $(0<p\le 2)$
 by stochastic $p$-stable processes,} 
Comment. Math. 48 (2008), no. 1, 3--12.


%\bibitem{GL74}
%Y. Gordon, D. Lewis,
%{\it Absolutely summing operators and local unconditional structures},
% Acta Math.  {133} (1974), 27--48.

 
%\bibitem{HS}
%F. Hernandez, E. Semenov,
%{\it Subspaces generated by translations in rearrangement invariant spaces,}
%J. Funct. Anal.  {169} (1999), 52--80.

\bibitem{HOS}
R. Haydon, E. Odell, T. Schlumprecht,
{\it
Small subspaces of $L_p$},
Ann. Math., 173 (2011),   169--209.
\bibitem{Holub}
J. 
Holub, {\it
On subspaces of separable norm ideals},
Bull. Amer. Math. Soc. 79 (1973), 446--448.

% \bibitem{HJSZ}
%J. Huang, M. Junge, F. Sukochev, D. Zanin,
%{\it Embeddings of noncommutative $L_p$-spaces into noncommutative symmetric spaces},
%preprint. 


\bibitem{HNSZ}
J. Huang, Y. Nessipbayev, F. Sukochev, D. Zanin,
{\it Compactness criteria in quasi-Banach symmetric operator spaces associated with a non-commutative torus,}
J. Funct. Anal. 289(5) (2025), Paper No. 110946, 35 pp. 

\bibitem{HSS}
J. Huang, O. Sadovskaya, F. Sukochev,
{\it On Arazy's problem concerning isomorphic embeddings of ideals of compact operators,}
 Adv. Math.  {406} (2022), Paper No. 108530, 21 pp.



\bibitem{HSSZ}
J. Huang, O. Sadovskaya, F. Sukochev, D. Zanin,
{\it 
Lack of isomorphic embeddings of $\ell_{p,q}$  into $L_{p,q}(\mathcal{M},\tau)$ 
over a noncommutative probability space,
}
Adv. Math. (2025), Paper No. 110571.

\bibitem{HS20}
J.Huang, F. Sukochev, {\it
Operator H\"older functions},
Banach J. of Math. Anal., 14(2020), 607--
629.

\bibitem{HS21}
J. Huang, F. Sukochev,
{\it Isomorphic classification of $L_{p,q}$-spaces, II},
J. Funct. Anal.  {280} (2021), 108994.  

\bibitem{HSZ}
J. Huang, F. 
 Sukochev, D.  Zanin, {\it  Operator $\theta$-H\"older functions with respect to $\norm{\cdot}_p$, $0<p\le \infty$,} 
  J. Lond. Math. Soc. (2) 105 (2022), no. 4, 2436--2477.


\bibitem{HSZ25}
J. Huang, F. Sukochev, D. Zanin,
{\it Isomorphisms between symmetric spaces over infinite and finite von Neumann algebras},
submitted manuscript. 

\bibitem{JSZ20}
Y. Jiao, F. Sukochev,  D. Zanin,
{\it Sums of independent and freely independent identically distributed random variables,}
Studia Math.  {255} (2020), 55--81.

\bibitem{JMST}
W. Johnson, B. Maurey, G. Schechtman, L. Tzafriri,
{\it Symmetric structures in Banach spaces, }
Mem. Amer. Math. Soc.   {19} (1979), no. 217, v+298 pp.

\bibitem{JS89}
W. Johnson, G. Schechtman,
{\it Sums of independent random variables in rearrangement invariant function spaces,}
Ann. Prob.  {17} (1989), 789--808. 

%\bibitem{JS}
%W. Johnson, G. Schechtman,
%{\it Embedding $\ell_p^m $ into $\ell_1^n$},
%Acta Math.  {149} (1982), 71--85.

%\bibitem{JS03}
%W. Johnson, G. Schechtman,
%{\it Very tight embeddings of subspaces of $L_p$, $1\le p<2$, into $\ell_{p}^n$},
%Geom. Funct. Anal.  {13} (2003), 845--851.

%\bibitem{JSZ17} 
%M. Junge, F. Sukochev, D. Zanin, 
%{\it 
%Embeddings of operator ideals into $L_p$-spaces on finite vonNeumann algebras},
% Adv. Math. 312 (2017), 473--546.

\bibitem{KP70}
M. Kadec, Pe{\l}czy\'{n}ski,
{\it Bases, lacunary sequences and complemented subspaces in the spaces $L_p$},
Studia Math.  {21} (1961/1962), 161--176.

\bibitem{KR} R. V. Kadison, J.  Ringrose, {\it Fundamentals of the theory of operator algebras. I},
Pure and Applied Mathematics, 100. Academic Press, Inc., New York, 1983.

%\bibitem{KMS}
%N. Kalton, S.  Montgomery-Smith,  {\it Set-functions and factorization,}
% Arch. Math. 61
% (1993), 183--200.


\bibitem{KPR}
N. Kalton, N. Peck, J. Roberts,
{\it An $F$-space sampler,}
Cambridge University Press,
Cambridge,
London--New York--New Rochelle--Melbourne--Sydney, 1984.

\bibitem{KS}
N. Kalton, F. Sukochev,
\textit{Symmetric norms and spaces of operators},
J. Reine Angew. Math.  621 (2008), 81--121.
   % \bibitem{Kam}
   % A.~Kaminska,
  %   \emph{Some remarks on Orlicz--Lorentz spaces}, Math. Nachr.
% {147}~(1990), no.~1, 29--38.
 
%\bibitem{KamMal}
%A.~Kaminska, L. Maligranda,
%{\it On Lorentz spaces $\Gamma_{p,\omega}$,}
%Israel J. Math.  {140} (2004), 285--318.

%\bibitem{KR}
%A.~Kaminska, Y. Raynaud,
%{\it Isomorphic copies in the lattice $E$ and its %symmetrization $E^{(*)}$ with applications to Orlicz--Lorentz spaces,}
%J. Funct. Anal.  {257} (2009), 271--331.

%\bibitem{Kosaki}
%H. Kosaki,
%{\it Unitarily invariant norms under which the map $A\to |A|$ is Lipschitz continuous, }
%Publ. RIMS, Kyoto Univ. 28 (1992), 299--313. 
 
 
 \bibitem{KR07}
A. Kamińska,  Y. Raynaud,  {\it  Copies of $l_p$ and $c_0$ in general quasi-normed Orlicz-Lorentz sequence spaces.} Function spaces, 207--227, Contemp. Math., 435, Amer. Math. Soc., Providence, RI, 2007. 
 
\bibitem{KPS}
S. Krein, Y. Petunin, E. Semenov,
 {\it Interpolation of linear operators},
 Translations of Mathematical Monographs, Amer. Math. Soc.
  {  54} (1982).

\bibitem{KurS}
    A. Kuryakov, F. Sukochev,
    {\it Isomorphic classification of $L_{p,q}$-spaces},
    J. Funct. Anal.  {269} (2015), 2611--2630.


\bibitem{KP}
 S. 
 Kwapien,   A. Pelczynski, {\it 
  The main triangle projection in matrix spaces and its
applications},  Studia Math. 34 (1970),  43--68.
 



% \bibitem{Lorentz50}
%  G.G. Lorentz, {\it Some new functional spaces}, Ann. Math. {  51} (1950), 37--55.
%  \bibitem{Lorentz51}
%   G.G. Lorentz, {\it On the theory of spaces $\Lambda $}, Pac. J. Math. 1 (1951) 411–429
 
    
    


%\bibitem{Bylinkina}
%O.P. Bylinkina,
%{\it Some estimates of the Banach--Mazur   distance between finite-dimensional $L_{p,q}$ spaces},
%Theory of operators in function spaces, 14--28,
%Voronezh. Gos. Univ. Voronezh, 1983. (Russian)
%
%\bibitem{Bogachev}
%V.I. Bogachev, {\it Measure theory,}
%Vol. I, II.
%Springer-Verlag, Berlin, 2007.


%\bibitem{Leung}
%D.H. Leung,
%{\it Isomorphism of certain weak $L^p$ spaces, }
%Studia Math.  {104} (1993), 151--160.
%
%\bibitem{Leung2}
%D.H. Leung,
%{\it Isomorphic classification of atomic weak $L^p$ spaces,}
%Interaction between functional analysis, harmonic analysis, and probability (Columbia, MO, 1994), 315--330,
%Lecture Notes in Pure and Appl. Math., 175, Dekker, New York, 1996.
%
%\bibitem{Leung3}
%D.H. Leung,
%{\it Purely non-atomic weak $L^p$ spaces,}
%Studia Math.  {122} (1997), 55--66.
%
%\bibitem{LS}
%D.H.  Leung, R. Sabarudin,
%{\it The classification problem for nonatomic weak $L^p$ spaces},
%J. Funct. Anal.  {258} (2010), 373--396.


%\bibitem{Lewis}
%D. Lewis,
%{\it An isomorphic characterization on the Schmidt class},
%Compositio Math.  {30} (1975), 293--297.
%\bibitem{LT}
 %   J. Lindenstrauss, L.  Tzafriri,
  %  {\it Classical Banach spaces,}
  %  Lecture Notes in Mathematics, 338, Springer-Verlag, Berlin--Hdidelberg--New York, 1973.

%\bibitem{LT3} J.~Lindenstrauss, L.~Tzafriri, {\it On Orlicz sequence spaces. III}, Israel J. Math.  {14} (1973), 368--389.

\bibitem{LP}
J. Lindenstrauss, A. Pelczynski,
{\it
Contributions to the theory of the classical Banach spaces},
J. Funct. Anal., 8 (1971),  225--249.


\bibitem{LT1}  J. Lindenstrauss, L.  Tzafriri,
 {\it Classical Banach spaces. I. Sequence spaces.}
 Ergebnisse der Mathematik und ihrer Grenzgebiete, Vol. 92. Springer-Verlag, Berlin-New York, 1977.

\bibitem{LT2} J.~Lindenstrauss, L.~Tzafriri, \emph{Classical Banach spaces. II. Function spaces,} Ergebnisse der Mathematik und ihrer Grenzgebiete,
     Vol. 97. Springer-Verlag, Berlin-New York, 1979.

 

\bibitem{LSZ}
S. Lord, F. Sukochev, D. Zanin,
 {\it Singular traces: Theory and applications,}
  De Gruyter Studies in Mathematics,  46. De Gruyter, Berlin, 2013.


\bibitem{Lorentz50}
G. Lorentz, {\it 
Some new function spaces}, Ann. of Math. 51 (1950),   37--55.
\bibitem{Lorentz51}
 G. Lorentz, {\it
 On the theory of spaces $\Lambda $}, Pacific J.
 Math. 1 (1951),   411--429.

\bibitem{Loza}
G.Ya. Lozanovski\u{ı}, {\it Mappings of Banach lattices of measurable functions}, Izv. Vyssh. Uchebn. Zaved. Mat. 5 (192) (1978),
84--86 (in Russian).

%\bibitem{LS}
%F. Lust-Piquard, F. Sukochev,
%{\it The $p$-Banach Saks property in symmetric operator %spaces,}
%Illinois J. Math. (51)4 (2007), 1207--1229.
%
%\bibitem{Luxemburg}
%W.A.J. Luxemburg,
%{\it Rearrangement-invariant Banach function spaces,}
%Proc. Symp. Analysis, Queen's univ.  (1967), 83--144.
%
%

\bibitem{Maligranda}
L. Maligranda,
{\it Type, cotype and convexity properties of quasi-Banach spaces,}
 in: M. Kato, L. Maligranda (Eds.), Banach
and Function Spaces, Proc. of the Internat. Symp. on Banach and Function Spaces, Kitakyushu--Japan, 2--4 Oct. 2003,
Yokohama Publ., 2004, pp. 83--120.

\bibitem{Mastylo}
M. Mastylo,
{\it On interpolation of some quasi-Banach spaces},
J. Math. Anal. Appl. 147 (1990), 403--419. 

%
%\bibitem{MP}
%B. Maurey, G. Pisier,
%{\it Series de variables al\'{e}atoires vectorielles independantes et propri\'{e}t\'{e}s geom\'{e}triques des espaces de Banach,}
%Studia Math.  {58} (1976), 45--90.


\bibitem{ON}
R. O'Neil,
{\it Integral transforms and tensor products on Orlicz spaces and $L(p,q)$ spaces,}
J. D'Analyse Math.  {21} (1968), 1--176.
 

\bibitem{Mc}
 Ch.  McCarthy, {\it $c_p$,}
  Israel J. Math. 5 (1967), 249--271.


\bibitem{MN}
P. 
Meyer-Nieberg, {\it Banach lattices},
 Universitext. Springer-Verlag, Berlin, 1991. 
 
\bibitem{Mityagin}
B. Mityagin, {\it 
The homotopy structure of the linear group of a Banach space}, Uspekhi Mat.Nauk 25 (5 (155)) (1970), 63--106 (in Russian); English transl. in: Russian Math. Surveys25 (5) (1970), 59--103.

%\bibitem{Nikisin}
%E. Nikisin, 
%{\it Resonance theorems and superlinear operators, }
%Uspekhi Mat. Nauk 25 (1970), 129--191. 


%\bibitem{Novikova}
%A. Novikova,
%{\it The Banach--Saks index for Rademacher subspaces,}
%Vestnik SSU, vol. 36, Samara University, Samara, 2005 (in Russian).

%\bibitem{ORP}
%S. Okada, W. Ricker,  E. Sanchez Perez, {\it 
%Optimal domain and integral extension of operators acting in function spaces,} Operator theory, Advances and Applications, vol. 180. Birkh\"auser, Basel, 2008.

\bibitem{Pisier}
G. 
Pisier, {\it
Some results on Banach spaces without local unconditional structure},
Compositio Math. 37 (1978), no. 1, 3--19.
 
 \bibitem{Ricard}
E. Ricard, {\it
Fractional powers on noncommutative $L_p$ for $p < 1$,} Adv. Math. 333 (2018), 194--211. 
 
%\bibitem{RS}
%V. Rodin, E. Semenov,
%{\it Rademacher series in symmetric spaces,}
%Anal. Math.  {1} (1975), 207--222.
 
   \bibitem{SS}
    O. Sadovskaya, F. Sukochev, {\it Isomorphic classification of $L_{p,q}$-spaces: the case $p=2$, $1\le q<2$,} Proc. Amer. Math. Soc. {  146} (2018),  3975--3984.

\bibitem{SemSuk}
E. Semenov, F. Sukochev,
{\it Sums and intersections of symmetric operator spaces},
J. Math. Anal. Appl. 414 (2014), 742--755. 

%\bibitem{Schutt}
%C. Schutt,
%{\it Lorentz spaces that are isomorphic to subspaces of $L_1$,}
%Trans. Amer. Math. Soc.  {314}(2) (1989), 583--595. 

 

\bibitem{S96}
F. Sukochev,
\textit{Non-isomorphism of $L_p$-spaces associated with finite and infinite von Neumann algebras,}
 Proc. Amer. Math. Soc. {124}(5)  (1996), 1517--1527.
 
 
\bibitem{S14}
F. Sukochev,
{\it Completeness of quasi-normed symmetric operator spaces},
Indag. 
Math. 25(2) (2014),  376--388. 

\bibitem{Sukochev16}
F. Sukochev, {\it H\"older inequality for symmetric operator spaces and trace property of K-cycles,}
 Bull. London Math. Soc. 48 (2016),  637--647

%\bibitem{Spos98} F. Sukochev, {\it RUC-bases in Orlicz and Lorentz operator spaces,} Positivity {  2} (1998), no. 3, 265--279.

\bibitem{Wojtaszczyk}
P. Wojtaszczyk, {\it
Banach spaces for analysts}, Cambridge University Press, Cambridge, 1991.

%\bibitem{YH}
%K. Yosida, E. Hewitt, {\it
% Finitely additive measures,}
% Trans. Am. Math. Soc. 72(1) (1952), 46–66.


%\bibitem{Zanin}
%D. Zanin,
%{\it Orbits and Khinchine-type inequalities in
%symmetric space},
%Ph.D. thesis, Flinders University, 2011. 

 
\end{thebibliography}
\end{document}